\documentclass[]{amsart}     
\usepackage{amsmath}
\usepackage{amsfonts}
\usepackage{amsthm}
\usepackage{graphicx}

\newtheorem{theorem}{Theorem}[section]

\newtheorem{proposition}[theorem]{Proposition}
\newtheorem{corollary}[theorem]{Corollary}
\newtheorem{lemma}[theorem]{Lemma}

\theoremstyle{remark}
\newtheorem{remark}[theorem]{Remark}
\theoremstyle{definition}

\DeclareMathOperator{\Cov}{Cov}
\DeclareMathOperator{\Var}{Var}
\DeclareMathOperator{\E}{\mathbb{E}}
\DeclareMathOperator{\cF}{\mathcal{F}}

\DeclareMathOperator{\bP}{\mathbb{P}}
\DeclareMathOperator{\ttau}{\tilde{\tau}}

\begin{document}

\title[Fractional Ornstein-Uhlenbeck with stochastic forcing]{Fractional Ornstein-Uhlenbeck process with stochastic forcing and its applications}

\author{Giacomo Ascione$^\ast$}
\address{$^\ast$ Dipartimento di Matematica e Applicazioni ``Renato Caccioppoli'', Universita degli Studi di Napoli Federico II, 80126 Napoli, Italy}
\author{Yuliya Mishura$^\odot $}
\address{$^\odot $ Department of Probability Theory, Statistics and Actuarial Mathematics, Taras Shevchenko National University of Kyiv, Volodymyrska 64, Kyiv 01601, Ukraine}
\author{Enrica Pirozzi$^\ast$}
\email{giacomo.ascione@unina.it \\
	myus@univ.kiev.ua \\
	enrica.pirozzi@unina.it}

\maketitle

\begin{abstract}
We consider a fractional Ornstein-Uhlenbeck process involving a stochastic forcing term in the drift, as a solution of a linear stochastic differential equation driven by a fractional Brownian motion. For such process we specify mean  and covariance functions, concentrating  on their asymptotic behavior. This  gives us a sort of short- or long-range dependence, under specified hypotheses on the covariance of the forcing process. Applications of this process in neuronal modeling are discussed, providing an example of a stochastic forcing term as a linear combination of Heaviside functions with random center. Simulation algorithms for the sample path of this process are finally given.
\end{abstract}
\section{Introduction}
In literature there are several processes that go by the name of fractional Ornstein-Uhlenbeck processes. Here, we refer to the fractional Ornstein-Uhlenbeck process of the first kind (fOU for short) introduced in \cite{CheriditoKawaguchiMaejima2003}. This process is the solution of a Langevin-type equation driven by a fractional Brownian motion (fBm for short) $B^H$
\begin{equation}\label{eq:lang}
dU_t^H=-\frac{1}{\theta} U^H_tdt+\sigma dB^H_t
\end{equation}
for some $\theta>0$ and $\sigma>0$.
A second kind of fractional Ornstein-Uhlenbeck process has been introduced in \cite{KaarakkaSalminen2011} and \cite{Kaarakka2015}, but we will not focus on it. Characteristic features   of fOU processes were also discussed in \cite{SunGuo2015}.
We add to \eqref{eq:lang}   an additional stochastic forcing term $I$, i.e., we consider an equation
\begin{equation}\label{eq:lang2}
dV_t=\left[-\frac{1}{\theta} V_t+I_t\right]dt+\sigma dB^H_t
\end{equation}
and   call its solution $V=\{V_t,t\ge 0\}$ a fractional Ornstein-Uhlenbeck process with forcing term $I=\{I_t, t\ge 0\}$ (ffOU, for short). The same process was studied in \cite{DehlingFrankeWoerner2017}, with a focus  on the case where $I$ is a periodic function. Observe that the ffOU in equation \eqref{eq:lang2}   depends on the interpretation of $dB^H_t$.\\
Fractional Brownian motion itself  has been widely studied in the last years (\cite{BiaginiHuOksendalZhang2008, Mishura2008, Nualart2006}). Its name is due to Mandelbrot and Van Ness, see \cite{MandelbrotVanNess1968}. A (two-sided) fBm with Hurst parameter $H \in (0,1)$ is an almost surely path-continuous centered Gaussian process $B^H=\{B^H_t, t \in \mathrm{R}\}$ with covariance function given by
\begin{equation*}
\E[B^H_tB^H_s]=\frac{1}{2}(|t|^{2H}+|s|^{2H}-|t-s|^{2H})
\end{equation*}
(\cite{MandelbrotVanNess1968, norros}). One-sided fBm is the process just introduced but restricted to  $t\ge 0$. In our paper we shall consider one-sided fBm, except when in subsection \ref{sub:cov} stationary processes starting from $t=-\infty$ are constructed. For $H=\frac{1}{2}$, fBm coincides with the standard Brownian motion. For $H\neq \frac{1}{2}$, fBm is  a non-Markovian process. Its increments are positively correlated as $H>\frac{1}{2}$, while   negatively correlated as $H<\frac{1}{2}$. The case $H>\frac{1}{2}$  describes the persistence of memory. The process $B^H $ admits stationary increments which are long-range dependent for $H>\frac{1}{2}$ and short-range dependent for $H<\frac{1}{2}$.
Integration with respect to the fBm can be introduced in several ways and the different resultant integrals are equal only for particular sets of integrands. An almost complete prospect of these integrals is given in \cite{BiaginiHuOksendalZhang2008}. All these integrals coincide  on particular deterministic functions, and for smooth integrands they coincide  with the limits of the Riemann-Stieltjes integral sums. In our case, referring to the path-wise approach, we will focus on Riemann-Stieltjes integrals.\\

Our main results are the representation and the study of the asymptotic behavior of the covariance of a ffOU process $V=\{V_t, t\ge 0\}$ with stochastic forcing term $I=\{I_t,t\ge 0\}$ in the case when $I_t \in L^2$ for any $t\ge 0$, and its covariance function is integrable w.r.t. the Lebesgue measure. We specify these results in the case of a particular stochastic forcing process useful in a neuronal modeling. The study of such process and its covariance function is very important while trying to define models with memory effects.  Indeed, we show that such processes preserve a sort of short- and long-range dependence for suitable forcing terms. For this reason, we show how such processes can be applied to neuronal modeling and how a forcing term can be chosen in order to describe some biological behaviors, such as neurons coupling and channel activation. In particular, we are interested in the memory effect which can be incorporated  into these models with the  help of the ffOU process $V $. Finally, we propose a simulation algorithm for this process. Simulations provide us a powerful tool to approximate numerically first passage time densities of the process through a constant threshold. These first passage times are important in neuronal modeling, since they describe the first spike time of the neuron on which the neuronal coding is based. In the classic case, these first passage times were also used to describe inter-spike intervals, since the process resets after a spike. In our case, this consideration could be done only if we also reset the memory of the process. Estimations of the densities of first passage time and inter-spike intervals will be considered in a future work. Inter-spike intervals have been already studied in \cite{RichardOrioTanre2017}, focusing such paper on stationariness of them.

Neuronal spiking modeling has been for about a century a central argument in mathematical physiology. One of the first attempt to describe neuronal activity is due to Lapicque in 1907, that introduced the classical integrate-and-fire model (\cite{Abbott1907}), from which the  Leaky integrate-and-fire  model (LIF for short) was derived. A LIF-type neuronal model is based on the following stochastic differential equation (SDE for short):
\begin{equation*}
dV_t=\frac{I_t}{C_m}dt-\frac{(V_t-\widetilde{V})}{\theta}dt+\sigma dB_t,
\end{equation*}
where $C_m$ is the membrane capacitance, $R_m$ is the membrane resistance, $I=I_t$ is the  current input, $\theta=C_mR_m$ is the characteristic time, $B=B_t$ is a Brownian motion, $\sigma$ a constant that is related to the intensity of the noise, $V=V_t$ is the membrane potential and $\widetilde{V}$ is the resting potential (\cite{KochSegev1998}). When the membrane potential upcrosses a particular threshold, the neuron fires and the process has to be reset to $\widetilde{V}$.\\
The crossings of this process through neuronal threshold generate spike trains of the potential. The first of these spikes has been, for instance, studied in \cite{BuonocoreCaputoPirozzi2008, BuonocoreCaputoPirozziRicciardi2011}. Successive spikes are instead studied, for instance, in \cite{DOnofrioPirozzi2016}. For the reconstruction of the input signal starting from some spike trains see, e.g., \cite{KimShinomoto2014}, and for   parameter estimation see, e.g., \cite{LanskyDitlevsen2008}.\\
However, the LIF model cannot completely describe the behavior of all neurons. In particular, in \cite{ShinomotoSakaiFunahashi1999}, it is shown that Ornstein-Uhlenbeck process is not adapted to data referring to neurons from the prefrontal cortex. One aspect the LIF model does not consider, is the adaptation of the neuron. Adaptation has been added in the model in various way. In \cite{BuonocoreCaputoCarforaPirozzi2016}, adaptation is introduced by coupling calcium dynamics with the LIF model equation. In \cite{TekaMarinovSantamaria2014, TekaUpadhayMondal2017}, adaptation is studied by using a fractional differential operator instead of a standard one, while in \cite{Pirozzi2017} it is introduced   by using a correlated noise instead of the white noise. The last approach is related to some integrated Gauss-Markov processes, whose first passage time densities are also investigated in \cite{Abundo2017}.\\
Our study has been widely inspired by the adaptation problem in neuronal modeling. For this reason, we focused on the application of the fOU process, which is a memory preserving process, in the context of neuronal modeling, describing the asymptotic behavior for the covariance as a sort of short- and long-range dependence.\\
In section \ref{sec-2} we describe the process $V=V_t$, providing a solution for the fractional Langevin equation   \eqref{eq:lang2}, that is adapted to the   filtration generated by its driving fBm and   the initial condition, under some particular hypothesis on $I$. In subsection \ref{sub:Exp} we also determine the mean  function for the case when $I_t \in L^1$, and its mean value  is integrable. In subsection \ref{sub:cov} we provide the covariance function when $I_t \in L^2$ and its covariance is integrable. We also study the covariance function as a function of $H$, determining how it changes as $H$ varies in $(0,1)$. In subsection \ref{sub:asy} some results of the study of the asymptotic behavior of the covariances are provided.\\
In section \ref{sec-3} we provide the neuronal model. We describe the process under a particular stochastic forcing term $I$ which is of biological interest, focusing on the mean and    covariance function, using the results of section \ref{sec-2}, in particular subsection \ref{sub:asy}.\\
In section \ref{sec-4} we provide two simulation algorithms that allow
 to obtain a method to simulate the first passage time density of the process through a constant threshold. The latter is important in the context of neuronal modeling, as it  describes the first spiking time of the neuron and can be used to study the inter-spike intervals.\\
The obtained theoretical results about mean  and covariance functions have allowed also graphical comparisons and consequent understanding of some properties of the process.
\section{The fractional Ornstein-Uhlenbeck process with forcing term}\label{sec-2}
\subsection{The stochastic Langevin equation with forcing term}\label{subsec2.1}
Let $(\Omega, \cF,\bP)$ be a complete probability space, $B^H=\{B^H_t,t\ge 0\}$ be a fBm   with Hurst parameter $H \in \left(0,\frac{1}{2}\right)\cup \left(\frac{1}{2},1\right)$, defined on this probability space,  and $\cF_t^H, t\ge 0$ be its natural filtration. We consider the following linear stochastic differential equation driven by our fBm:
\begin{equation}\label{eq:fdiffeq}
dV_t=\left[-\frac{1}{\theta}(V_t-\widetilde{V})+I_t\right]dt+\sigma dB^H_t, \
V_0=\xi
\end{equation}
where $\xi$ is a square-integrable random variable defined on $(\Omega, \cF,\bP)$, $I=I_t$ is a forcing term, $\theta>0$ and $\widetilde{V}$ are constant. In particular, $I$ can be a stochastic process defined on $(\Omega, \cF, \bP)$ (eventually degenerate), or on a different probability space $(\Omega', \cF', \bP')$. Equation \eqref{eq:fdiffeq} is a slight modification of \eqref{eq:lang2} and will be our fractional Langevin equation with forcing term $I$; it is equivalent to the integral equation
\begin{equation}\label{eq:finteq}
V_t=\xi-\frac{1}{\theta}\int_{0}^{t}(V_s-\widetilde{V})ds+\int_{0}^{t}I_sds+\sigma B^H_t.
\end{equation}
We call the solution $V=V_t$ of equation \eqref{eq:fdiffeq} (or, equivalently, \eqref{eq:finteq}) ffOU process with forcing term $I=I_t$ and $\widetilde{V}$ its resting term, that is to say a globally asymptotically stable equilibrium for the expected value as $I_t \equiv 0$.
\subsection{The fractional Ornstein-Uhlenbeck process}
Taking into account equation \eqref{eq:finteq}
and following the lines of proposition A.1 of \cite{CheriditoKawaguchiMaejima2003}, we can explicitly determine the solution of \eqref{eq:finteq}, and so to present an explicit  form of a ffOU with a forcing term $I$. In order to do this, denote $\cF_t^{H,\xi}$ the sigma-filed generated by $\xi$ and $B^H_s, s\le t$.
\begin{proposition}\label{thm:solSDE}
With the notation shown in subsection \ref{subsec2.1}, if $I$ is a stochastic process such that its sample paths are (almost surely)  integrable in any interval $[0,T]$ for $T>0$, then equation \eqref{eq:finteq} admits a unique solution whose paths are almost surely continuous. This solution can be expressed as
\begin{equation}\label{eq:proc}
V_t=\widetilde{V}+e^{-\frac{t}{\theta}}\left(-\widetilde{V}+\xi+\int_{0}^{t}I_se^{\frac{s}{\theta}}ds+\sigma \int_{0}^{t}e^{\frac{s}{\theta}}dB^H_s\right),
\end{equation}
where the integral is point-wise interpreted as  a Riemann-Stieltjes integral. In the case $H<1/2$ it is defined via integration by parts, namely,
\begin{equation}\label{eq:proc1}
V_t=\widetilde{V}+e^{-\frac{t}{\theta}}\left(-\widetilde{V}+\xi+\int_{0}^{t}I_se^{\frac{s}{\theta}}ds+\sigma \left(e^{\frac{t}{\theta}} B^H_t-\theta^{-1}\int_{0}^{t}e^{\frac{s}{\theta}} B^H_sds\right)\right).
\end{equation}
Moreover, $V_t$ is $\cF_t^{H,\xi}$-adapted if $I_t$ is a $\cF_t^{H,\xi}$-adapted process on $(\Omega, \cF, \bP)$.
\end{proposition}

\begin{remark}
Note that the integrability of the sample paths of $I$ is sufficient to guarantee the integrability of the paths of $e^{\frac{t}{\theta}}I_t$, because of the obvious upper bound
\begin{equation*}
\int_0^te^{\frac{s}{\theta}}|I_s|ds \le e^{\frac{t}{\theta}}\int_0^{t}|I_s|ds.
\end{equation*}
\end{remark}

\begin{remark}
From the point of view of applications to neuronal modeling, in some cases it is natural to assume that  $I$ is a stochastic process defined on the different probability space $(\Omega', \cF', \bP')$. Then the solution $V=V_t$ is defined on the product probability space $(\Omega\times\Omega',\cF\otimes\cF',\bP\otimes \bP') $. Let us denote the solution as $V_t(\omega, \omega')$. Thus, for fixed $\omega'\in \Omega'$, the process $X_t=V_t(\cdot,\omega')$ is $\cF_t^{H,\xi}$-adapted. In fact, for fixed $\omega'$ we have  from equation \eqref{eq:proc}
\begin{equation*}
X_t=\widetilde{V}+e^{-\frac{t}{\theta}}\left[-\widetilde{V}+\xi+\int_{0}^{t}I_s(\omega')e^{\frac{s}{\theta}}ds+\sigma \int_{0}^{t}e^{\frac{s}{\theta}}dB^H_s\right]
\end{equation*}
where $s\mapsto I_s(\omega')$ is a deterministic function.
\end{remark}

The similar result to proposition \ref{thm:solSDE} is also shown in \cite{DehlingFrankeWoerner2017} (see proposition 2.1) by using It\^{o} formula for divergence integral established in \cite{Nualart2006}. In proposition 2.2 of \cite{DehlingFrankeWoerner2017}, the stationary solution is determined for a periodic deterministic forcing term $I(t)$, which is  possible whenever $\int_{-\infty}^{t}e^{\frac{s}{\theta}}I_sds<+\infty$.\\
\subsection{The mean value function}\label{sub:Exp}
Using  proposition \ref{thm:solSDE} and recalling that $$\mathbb{E}\left(\int_{0}^{t}e^{\frac{s}{\theta}}dB^H_s\right)=0$$ for any $t>0$,
we immediately get the following result.
\begin{proposition}\label{prop:mean}
With the notation specified in subsection \ref{subsec2.1}, let $I$ be a stochastic process defined on $(\Omega, \cF, \bP)$ such that
 for any $t>0$, $I_t \in L^1(\Omega, \cF, \bP)$, and $\E[|I_t|]\in L^1([0,T],\lambda)$, where $\lambda$ is the Lebesgue measure.
 Also,  let $V=V_t$ be the solution of equation \eqref{eq:fdiffeq}.
Then
\begin{equation}\label{eq:meangen}
\E[V_t]=(1-e^{-\frac{t}{\theta}})\widetilde{V}+e^{-\frac{t}{\theta}}\E[\xi]+e^{-\frac{t}{\theta}}\int_0^{t}e^{\frac{s}{\theta}}\E[I_s]ds.
\end{equation}
\end{proposition}

\begin{remark}
Obviously, the mean value function is a solution to the following Cauchy problem:
\begin{equation*}
\begin{cases}
\dot{m}(t)=-\frac{1}{\theta}(m(t)-\widetilde{V})+\E[I_t], \\
m(0)=\E[\xi].
\end{cases}
\end{equation*}
Moreover, if $I_t \equiv 0$, then  $\lim_{t \to +\infty}m(t)=\widetilde{V}$ (recall that $\theta>0$). In the general case asymptotics of $m(t)$ depends on the asymptotics of $\E[I_t]$. For example, if $\E[I_t]$ is a continuous function, and $\lim_{t \to +\infty}\E[I_t]=A\in \mathrm{R}$, then $\lim_{t \to +\infty}m(t)=\widetilde{V}+\theta A$, according to the L'Hospital's rule.
\end{remark}

\begin{remark}
If $I_s$ is a deterministic integrable function on $[0,t]$ for any $t>0$, then
\begin{equation}\label{eq:meandet}
\E[V_t]=(1-e^{-\frac{t}{\theta}})\widetilde{V}+e^{-\frac{t}{\theta}}\E[\xi]+e^{-\frac{t}{\theta}}
\int_0^{t}e^{\frac{s}{\theta}}I_sds.
\end{equation}

\end{remark}

\begin{remark}
If $I$ is a stochastic process defined on a different probability space $(\Omega', \cF', \bP')$, let us denote $\E$ as the expectation taken on $(\Omega, \cF, \bP)$. If $I$ has sample paths which are integrable in any interval $[0,T]$ for $T>0$, then $\E[V_t]$ is a stochastic process on $(\Omega',\cF',\bP')$. Indeed for fixed $\omega' \in \Omega'$, $s \mapsto I_s(\omega')$ is a deterministic function and then equation \eqref{eq:meandet} holds, becoming
\begin{equation*}
\E[V_t(\cdot,\omega')]=(1-e^{-\frac{t}{\theta}})\widetilde{V}+e^{-\frac{t}{\theta}}\E[\xi]+e^{-\frac{t}{\theta}}\int_0^{t}e^{\frac{s}{\theta}}I_s(\omega')ds.
\end{equation*}
In this case, we do not need any hypothesis on the mean function of $I$, but only the integrability of its paths.
\end{remark}

\subsection{The covariance and variance  functions. Analytic formulas }\label{sub:cov}
Bearing in mind that the main feature of the process $V$ is revealed by its covariance function, we calculate it for zero  forcing term and expand to non-zero case.  Based on this, it is easy to calculate a variance function. There are two approaches to the calculation of covariance function of the fractional Ornstein-Uhlenbeck process. One of them is based on its harmonizable representation, another one is based on its representation as the Wiener integral w.r.t. the fBm.
\subsubsection{Covariance function via harmonizable representation}  In order to calculate  the covariance function in terms of the harmonizable representation,    we   introduce the stationary fractional Ornstein-Uhlenbeck process $U^H$ with parameter $\frac{1}{\theta}$  (sfOU for short), as described in \cite{CheriditoKawaguchiMaejima2003}. This process is the unique solution to the equation
\begin{equation}\label{eq:langfrac}
dU^H_t=-\frac{1}{\theta}U^H_tdt+\sigma dB^H_t,
\end{equation}
that is   a stationary Gaussian process.
According to \cite{CheriditoKawaguchiMaejima2003}, a general solution $U^{H,\eta}_t$ of this equation with initial value $U^{H,\eta}_0=\eta$, where $\eta$ is a square integrable random variable, has the form
\begin{equation*}
U^{H,\eta}_t=e^{-\frac{t}{\theta}}\left(\eta+\sigma\int_{0}^{t}e^{\frac{s}{\theta}}dB^H_s\right), \  t\ge 0
\end{equation*}
and the only stationary solution of equation \eqref{eq:langfrac} admits $\eta=\sigma\int_{-\infty}^{0}e^{\frac{s}{\theta}}dB^H_s$. Hence
\begin{equation*}
U_t^H=\sigma \int_{-\infty}^{t}e^{-\frac{t-s}{\theta}}dB^H_s
\end{equation*}
which is a stationary long-range dependent process for $H>\frac{1}{2}$, while is short-range dependent for $H<\frac{1}{2}$.\\
A closed form of its covariance function for $H\in(0,1)$   was obtained  in \cite{PipirasTaqqu2000}, see also   \cite{CheriditoKawaguchiMaejima2003}, Remark 2.4. It has  a form
\begin{equation}\label{eq:sfOUCov}
\rho(s):=\Cov(U^H_t,U^H_{t+s})=\sigma^2\theta^2C_H\int_{-\infty}^{+\infty}\frac{|x|^{1-2H}}{1+\theta^2x^2}e^{isx}dx,
\end{equation}
where $C_H=\frac{\Gamma(2H+1)\sin(\pi H)}{2\pi}$. It is also shown in \cite{CheriditoKawaguchiMaejima2003} that, for $s \to \infty$, the  covariance function admits the following asymptotic expression for any $H \in (0,1)\setminus\{1/2\}$, $N=1,2,\dots$ and for fixed $t$:
\begin{equation*}
\Cov(U^H_t,U^H_{t+s})=\frac{1}{2}\sigma^2\sum_{n=1}^{N}\theta^{2n}\left(\prod_{k=0}^{2n-1}(2H-k)\right)s^{2H-2n}+O(s^{2H-2N-2}).
\end{equation*}
In particular, for $N=1$, we obtain
\begin{equation}\label{covcov}
\Cov(U^H_t,U^H_{t+s})=\sigma^2\theta^2H(2H-1)s^{2H-2}+O(s^{2H-4}).
\end{equation}
For $\frac{1}{2}<H<1$, we have $-1<2H-2<0$, then $\Cov(U^H_t,U^H_{t+s})\sim K s^{2H-2}$ as $s \to +\infty$, for some constant $K$,   saying, as usual,  that $f \sim g$  if $\lim_{s \to +\infty}\frac{f(s)}{g(s)}=1$. Hence the long-range dependence of the process $U^H$ for $H>\frac{1}{2}$ follows. Conversely, for $0<H<\frac{1}{2}$ we have $-2<2H-2<-1$ whence the short-range dependence follows.

Using these results and   Theorem $2.3$ in \cite{CheriditoKawaguchiMaejima2003} we can get the   covariance function in the following symmetric form.
\begin{lemma}\label{lem:fOU}
Let $U^{H,x}=U^{H,x}_t, t\ge 0$ be a fOU process solving equation \eqref{eq:langfrac} with initial value $x \in \mathbb{R}$.   Then its covariance does not depend on $x\in \mathrm{R}$ and has a form
\begin{equation}\begin{gathered}\label{eq:covUHX}
R_H(t,  s):=\Cov(U^{H,x}_t,U^{H,x}_{ s})\\=\sigma^2\theta^2C_H\int_{\mathrm{R}}\left(e^{isy}-e^{-\frac{s}{\theta}}\right)\left(\overline{e^{ity}-e^{-\frac{t}{\theta}}}\right)\frac{|y|^{1-2H}}{1+\theta^2y^2}dy.
\end{gathered}\end{equation}
\end{lemma}
\begin{proof}
First note that
\begin{equation*}
U^{H,x}_t-\mathbb{E}[U^{H,x}_t]=\sigma \int_{0}^{t}e^{-\frac{t-v}{\theta}}dB^H_v.
\end{equation*}
Let $s\ge t$. Then we have
\begin{equation}\label{eq:pass2}
\begin{gathered}
R_H(t,  s) =\sigma^2\E\left[\left(  \int_{0}^{t}e^{-\frac{t-v}{\theta}}dB^H_v\right)\left( \int_{0}^{ s}e^{-\frac{ s-u}{\theta}}dB^H_u\right)\right]\\
 =\sigma^2\bigg(\E\left[ \left(\int_{-\infty}^{t}e^{-\frac{t-v}{\theta}}dB^H_v-\int_{-\infty}^{0}e^{-\frac{t-v}{\theta}}dB^H_v\right)\right.\\
 \times\left. \left(\int_{-\infty}^{ s}e^{-\frac{ s-u}{\theta}}dB^H_u-\int_{-\infty}^{0}e^{-\frac{ s-u}{\theta}}dB^H_u\right)\right]\bigg)\\
 =\sigma^2\bigg(\E\left[ \left( \int_{-\infty}^{t}e^{-\frac{t-v}{\theta}}dB^H_v\right)\left(  \int_{-\infty}^{ s}e^{-\frac{ s-u}{\theta}}dB^H_u\right)\right]\\
 - \E\left[ \left(\int_{-\infty}^{0}e^{-\frac{t-v}{\theta}}dB^H_v\right)\left(  \int_{-\infty}^{ s}e^{-\frac{ s-u}{\theta}}dB^H_u\right)\right]\\
 -\ \E\left[ \left(\int_{-\infty}^{t}e^{-\frac{t-v}{\theta}}dB^H_v\right)\left( \int_{-\infty}^{0}e^{-\frac{ s-u}{\theta}}dB^H_u\right)\right]\\
 + \E\left[ \left(\int_{-\infty}^{0}e^{-\frac{t-v}{\theta}}dB^H_v\right)\left( \int_{-\infty}^{0}e^{-\frac{ s-u}{\theta}}dB^H_u\right)\right]\bigg)\\
  =\rho(s-t)-e^{-\frac{t}{\theta}}\rho( s)-e^{-\frac{ s}{\theta}}\rho(t)+e^{-\frac{ t+s}{\theta}}\rho(0).
\end{gathered}
\end{equation}

It follows immediately from the last line of \eqref{eq:pass2} and symmetry of $R_H$ w.r.t. $t$ and $s$ as well as the symmetry of the integrals involved, that
\begin{equation*}\begin{gathered}
R_H(t,  s)=\sigma^2\theta^2C_H\bigg(\int_{\mathrm{R}}\frac{|y|^{1-2H}}{1+\theta^2y^2}e^{i(s-t)y}dy
 -\int_{\mathrm{R}}\frac{|y|^{1-2H}}{1+\theta^2y^2}e^{i sy-\frac{t}{\theta}}dy\\-
 \int_{\mathrm{R}}\frac{|y|^{1-2H}}{1+\theta^2y^2}e^{ity-\frac{s}{\theta}}dy
 +e^{-\frac{t+s}{\theta}}\int_{\mathrm{R}}\frac{|y|^{1-2H}}{1+\theta^2y^2}dy\bigg)\\
  =\sigma^2\theta^2C_H\int_{\mathrm{R}}\left(e^{isy}(e^{-ity}-e^{-\frac{t}{\theta}})-e^{-\frac{ s}{\theta}}(e^{-ity}-e^{ -\frac{t}{\theta}})\right)\frac{|y|^{1-2H}}{1+\theta^2y^2}dy\\=
  \sigma^2\theta^2C_H\int_{\mathrm{R}}\left(e^{isy}-e^{-\frac{s}{\theta}}\right)\left(\overline{e^{ity}-e^{-\frac{t}{\theta}}}\right)\frac{|y|^{1-2H}}{1+\theta^2y^2}dy.
\end{gathered}\end{equation*}
\qed
\end{proof}

\subsubsection{Covariance function via the  representation  of fOU as the Wiener integral w.r.t. the fBm} Concerning the  calculation of the covariance function of the fOU process  based on its representation as the Wiener integral w.r.t. fBm, we recall that for $H>1/2$ $$\E\left(\int_0^sf(u)dB^H_u\int_0^tf(v)dB^H_v\right)=H(2H-1)\int_0^s\int_0^tf(u)f(v)|u-v|^{2H-2}dvdu,$$
 for any measurable function $f$, for which the right-hand side of this equality is well defined, while for $H<1/2$ and any continuous function $f$ of bounded variation we have that $$\int_0^sf(u)dB^H_u=B^H_sf(s)-\int_0^sB^H_udf(u).$$
 With the help of these facts, and applying the representation of covariance function from \cite{Mishura2017}, we can write it  for any $H \in (0,1)$ and $t\ge s\ge 0$ in the following non-symmetric w.r.t. $s$ and $t$ form that permits to avoid the absolute values of the time differences:
\begin{equation}\label{covcov1}
\begin{gathered}
R_H(t,s)=\frac{H\sigma^2}{2}\left(-e^{\frac{ s-t}{\theta}}\int_0^{t-s}e^{\frac{z}{\theta}}
z^{2H-1}dz+e^{\frac{t-s}{\theta}}\int_{t-s}^{t}e^{-\frac{z}{\theta}}z^{2H-1}dz\right.\\\left. -e^{-\frac{t+s}{\theta}}\int_{s}^{t}e^{\frac{z}{\theta}}z^{2H-1}dz+e^{ \frac{ s-t}{\theta}}
\int_{0}^{s}e^{-\frac{z}{\theta}}z^{2H-1}dz +2e^{-\frac{t+s}{\theta}}\int_0^{t}e^{\frac{z}{\theta}}z^{2H-1}dz\right).
\end{gathered}
\end{equation}
 \subsubsection{Covariance function of ffOU} Now we proceed with the covariance function for ffOU process, i.e., fOU process with non-zero forcing term $I$.  Denote $c(u,v)=\Cov(I_u,I_v)$. From now on, we assume that the initial value $\xi=x\in\mathrm{R}$. The proof of the following proposition immediately follows from the representation \eqref{eq:proc}.
\begin{lemma}\label{prop:covgen}
With the notation specified in subsection  \ref{subsec2.1}, let $\xi$ be a degenerate random variable. Suppose $I$ is a stochastic process defined on $(\Omega, \cF, \bP)$ such that:
\begin{itemize}
 \item[$(i)$] For any $t>0$, $I_t \in L^2(\Omega, \cF, \bP)$.
\item[$(ii)$] For any $T>0$, $t \mapsto \E[|I_t|]$ is integrable in $[0,T]$.
\item[$(iii)$] For any $(t,s) \in \mathbb{R}^2$ such that $t,s>0$, $(u,v)\mapsto c(u,v)$ is integrable in $[0,t]\times [0,s]$.
\item[$(iv)$] For any $t,s>0$ random variables  $I_t$ and $B^H_s$ are uncorrelated.
\end{itemize}
Then
\begin{align}
\label{eq:covV}
\begin{split}
\Cov(V_t,V_{s})&=e^{-\frac{t+s}{\theta}}\iint_{[0,t]\times[0,s]}e^{\frac{u+v}{\theta}}c(u,v)dudv+R_H(t,s).
\end{split}
\end{align}
\end{lemma}

\begin{remark}
From \eqref{eq:covUHX}  we get the following representation of the covariance function of $V_t$:
\begin{equation*}
\begin{gathered}
\Cov(V_t,V_{s})=e^{-\frac{t+s}{\theta}}\iint_{[0,t]\times[0,s]}e^{\frac{u+v}{\theta}}c(u,v)dudv\\
+\sigma^2\theta^2C_H\int_{\mathrm{R}}\left(e^{isy}-e^{-\frac{s}{\theta}}\right)\left(\overline{e^{ity}-e^{-\frac{t}{\theta}}}\right)\frac{|y|^{1-2H}}{1+\theta^2y^2}dy.
\end{gathered}
\end{equation*}
\end{remark}
\begin{remark}
If $I$ and $B^H$ are correlated stochastic processes, then we denote $J_t=\int_{0}^{t}e^{\frac{s}{\theta}}I_sds$  and get two cross covariances $\Cov(J_t,U^{H,x}_{ s})$ and $\Cov(U^{H,x}_t,J_{ s})$ that cannot be neglected. In particular,   notice that
 \begin{align*}
\Cov(U^{H,x}_t,J_{ s})=\E[U^{H,0}_tJ_{ s}] && \Cov(J_t,U^{H,x}_{ s})=\E[J_tU^{H,0}_{ s}].
\end{align*}
Using this observation, we can obtain a more general formula
\begin{align*}
\begin{split}
\Cov(V_t,V_{s})&=e^{-\frac{t+s}{\theta}}\iint_{[0,t]\times[0,s]}e^{\frac{u+v}{\theta}}c(u,v)dudv+R_H(t, s)\\&+e^{-\frac{t}{\theta}}\E[J_tU^{H,0}_{ s}]+e^{-\frac{t+s}{\theta}}\E[U^{H,0}_{t}J_{ s}].
\end{split}
\end{align*}
\end{remark}

\begin{remark}\label{rmk:cov}
If $I$ is a deterministic function, then the covariance function of the process $V_t$ does not depend on $I$ (since $c(u,v)=0$ for all $u,v>0$), so it coincides with the covariance function of the fOU process with deterministic initial value and parameter $\frac{1}{\theta}$, that is to say $\Cov(V_t,V_{ s})=R_H(t, s)$ and then $\Cov(V_t,V_{ +s})\sim Ks^{2H-2}$, preserving the long-range dependence and the short-range dependence respectively for $H>\frac{1}{2}$ and $H<\frac{1}{2}$.
This also happens if $I$ is a stochastic process defined on a different probability space $(\Omega',\cF',\bP')$. It is important to notice that  in such case  $\Cov(V_t,V_{ t+s})$, calculated for any fixed $\omega'\in\Omega'$  w.r.t. the measure $\bP$ is  still a deterministic function.
\end{remark}
\subsubsection{Variance function}
Let us put $s=t$ in equation \eqref{eq:covV} and get, under the   hypothesis of proposition \ref{prop:covgen}, the following variance function
\begin{align}\label{eq:vargen}\begin{split}
\Var[V_t]=R_H(t,t)+e^{-\frac{2t}{\theta}}\iint_{[0,t]^2}e^{\frac{v+u}{\theta}}c(u,v)dudv.
\end{split}
\end{align}
Moreover, if $I$ and $B^H$ are correlated, then we obtain the more general formula
\begin{align}\label{eq:vargencorr}\begin{split}
\Var[V_t]&=R_H(t,t)+e^{-\frac{2t}{\theta}}\iint_{[0,t]^2}e^{\frac{v+u}{\theta}}c(u,v)dudv\\&+2e^{-\frac{t}{\theta}}\E[J_tU^{H,0}_t]
\end{split}
\end{align}
with $J_t$ defined as before.
With the  harmonizable representation,  the   formulas \eqref{eq:vargen} and \eqref{eq:vargencorr} become, respectively,
\begin{align*}\begin{split}
\Var[V_t]&=\sigma^2\theta^2C_H\int_{\mathrm{R}}| e^{ity}-e^{-\frac{t}{\theta}}|^2\frac{|y|^{1-2H}}{1+\theta^2y^2}dy\\&
+e^{-\frac{2t}{\theta}}\int_{[0,t]^2}e^{\frac{v+u}{\theta}}c(u,v)dudv,
\end{split}
\end{align*}
and
\begin{align*}\begin{split}
\Var[V_t]&=\sigma^2\theta^2C_H\int_{\mathrm{R}}| e^{ity}-e^{-\frac{t}{\theta}}|^2\frac{|y|^{1-2H}}{1+\theta^2y^2}dy\\&+e^{-\frac{2t}{\theta}}\iint_{[0,t]^2}e^{\frac{v+u}{\theta}}c(u,v)dudv+2e^{-\frac{t}{\theta}}\E[J_tU^{H,0}_t].
\end{split}
\end{align*}
In particular, if $I$ is a deterministic function, then simply
\begin{equation}\label{eq:detvar}
\Var[V_t]=\sigma^2\theta^2C_H\int_{\mathrm{R}}| e^{ity}-e^{-\frac{t}{\theta}}|^2\frac{|y|^{1-2H}}{1+\theta^2y^2}dy.
\end{equation}

\subsection{Covariance and variance as the functions of time and Hurst index. Asymptotic behavior}\label{sec-cov}
\subsubsection{Asymptotic behavior of covariance. Non-random forcing term}\label{subsec-cov-1}
It is easy to see that both representations \eqref{eq:covUHX} and \eqref{covcov1} for the covariance function of the fOU process (and even more of the ffOU process) are not so simple as to immediately analyze their behavior regarding parameters $t$, $s$ and $H$. Therefore we will analyze their asymptotics with respect to these parameters. To start, note that in \cite{CheriditoKawaguchiMaejima2003} the authors proved the following asymptotic expansion for fixed $t \ge 0$, $N=1,2,\dots$ and $s \to +\infty$:
\begin{equation*}\begin{gathered}
R_H(t,t+s)=\frac{1}{2}\sigma^2\sum_{n=1}^{N}\theta^{2n}\left(\prod_{k=0}^{2n-1}(2H-k)\right)\left[s^{2H-2n} -e^{-\frac{t}{\theta} }(t+s)^{2H-2n}\right]\\+O(s^{2H-2N-2}),
\end{gathered}\end{equation*}
which, for $N=1$, becomes
\begin{equation*}
R_H(t,t+s)=\sigma^2\theta^{2}H(2H-1)(s^{2H-2}-e^{-\frac{t}{\theta}}(t+s)^{2H-2})+O(s^{2H-4})
 \end{equation*}
so that also $R_H(t,t+s)\sim Ks^{2H-2}$ for some constant $K$. For this reason we can conclude that $U_t^{H,x}$, as well as  $U_t^{H}$, demonstrates a time non-homogeneous long-range dependence for $H>\frac{1}{2}$ and a time non-homogeneous short-range dependence for $H<\frac{1}{2}$.

Concerning the value of the constant $K$,  Figure \ref{CovSigmabasso} on the left demonstrates that the tails of $s \mapsto R_H(t,t+s)$ as $s\rightarrow\infty$   depend on the Hurst parameter and they are slower in convergence to $0$ as $H$ grows.
Let us investigate the asymptotic behavior of the covariance function at the boundaries, i.e., as $H\to 1$ and $H\to 0.$ Recall that for $H=1$ we have that $B^1_t=t\xi$, where $\xi=\mathcal{N}(0,1)$. Therefore,   for $H=1$ and any continuous function $f$,  it holds that
$\int_0^tf(s)dB^1_s=\xi\int_0^tf(s)ds$. It means that
$$\int_0^te^{\frac{s}{\theta}}dB^1_s=\xi\int_0^te^{\frac{s}{\theta}}ds=\xi\theta (e^{\frac{t}{\theta}}-1).$$ To construct fOU process with $H=0$, consider   $\eta=\{\eta_{t}, t\ge 0\}$ the Gaussian white noise process with variance $\frac12$, i.e.\ the $\eta_t$'s are i.i.d.\ $\mathcal{N}(0, \frac12)$
--distributed random variables. Set
\begin{equation}\label{whiteno}
B^0_t: = \eta_t - \eta_0, \quad t\ge 0.
\end{equation}
It was proved in \cite[Lemma $4.1$]{novimish} that the finite-dimensional distributions of $B^H$ converge weakly to the finite-dimensional distributions of $B^0$ as $H\to 0.$  Therefore we put the fBm at zero to be equal $B^0$, and respectively,
\begin{equation}\label{eq:oterpass}
\int_0^te^{\frac{s}{\theta}}dB^0_s= B^0_te^{\frac{t}{\theta}}-\theta^{-1}\int_0^te^{\frac{s}{\theta}} B^0_sds,
\end{equation}
where $B^0$ is taken from \eqref{whiteno}.

\begin{theorem}\label{lem:CovH}
Let a forcing term $I$ and an initial value be non-random. Then
\begin{itemize}
\item[$(i)$] The function $\Cov(V_t,V_{s})=R_H( t, {s})$ as a function of $H\in (0,1)$ and $(s,t)\in \mathrm{R}^2_+$ is continuous on $(0,1)\times \mathrm{R}^2_+$.
\item[$(ii)$]   For any $(s,t)\in R^2_+$
\begin{equation}\begin{gathered}\label{limcov}
\lim_{H \to 1} R_H( t, {s})= \sigma^2\theta^{2}\left(1- e^{-\frac{s}{\theta}}\right)\left(1-e^{-\frac{t}{\theta}}\right).
\end{gathered}\end{equation}
This result coincides with the formula that can be obtained if we directly substitute  $B^1_t=t\xi$  into  equality \eqref{eq:proc}. So, if we put  $$R_1( t, {s})= \sigma^2\theta^{2}\left(1- e^{-\frac{s}{\theta}}\right)\left(1-e^{-\frac{t}{\theta}}\right),$$
then  $R_{H}( t, {s})$ becomes continuous on $(0,1]\times \mathrm{R}^2_+$.
\item[$(iii)$]   For any $(s,t)\in \mathrm{R}^2_+$
\begin{equation}\begin{gathered}\label{limcov2}
 \lim_{H \to 0} R_H( t, {s})=
\begin{cases}
 0,\; \text{if}\; s\wedge t=0;\\
\frac{\sigma^2}{2}e^{-\frac{t+s}{\theta}}\;\text{if}\; s\not=t\mbox{ and }s\wedge t>0;\\
\frac{\sigma^2}{2}\left(1+e^{-\frac{2t}{\theta}}\right)\; \text{if}\;s=t>0.
\end{cases}
\end{gathered}\end{equation}
This result coincides with the formula that can be obtained if we directly substitute $B^0_t=\eta_t$  into  equality \eqref{eq:proc}. So, if we put
\begin{equation*}\begin{gathered}
R_0( t, {s})=
\begin{cases}
 0,\; \text{if}\; s\wedge t=0;\\
\frac{\sigma^2}{2}e^{-\frac{t+s}{\theta}}\;\text{if}\; s\not=t\mbox{ and }s\wedge t>0;\\
\frac{\sigma^2}{2}\left(1+e^{-\frac{2t}{\theta}}\right)\; \text{if}\;s=t>0.
\end{cases}
\end{gathered}\end{equation*}
$R_{H}( t, {s})$ is a continuous function on $[0,1] $ for any fixed $t,s$. However, as a function of the three variables $(H,t,s)$, it is discontinuous at point $H=0$ on the axes $s=0$ and $t=0$ and on the straight line $t=s$.
\end{itemize}
\end{theorem}
\begin{proof} Item $(i)$ is evident. Concerning $(ii)$, using dominated convergence theorem, we can go to the limit as $H\to 1$ in \eqref{covcov1} and get that
 \begin{equation}\begin{gathered}\label{limcov1}
\lim_{H \to 1} R_H( t, {s})=\frac{\sigma^2}{2}\left(-e^{\frac{ s-t}{\theta}}\int_0^{t-s}e^{\frac{z}{\theta}}
z dz+e^{\frac{t-s}{\theta}}\int_{t-s}^{t}e^{-\frac{z}{\theta}}z dz\right.\\\left. -e^{-\frac{t+s}{\theta}}\int_{s}^{t}e^{\frac{z}{\theta}}z dz+e^{ \frac{ s-t}{\theta}}
\int_{0}^{s}e^{-\frac{z}{\theta}}z dz +2e^{-\frac{t+s}{\theta}}\int_0^{t}e^{\frac{z}{\theta}}z dz\right).
\end{gathered}\end{equation}
Integrating in all integrals in \eqref{limcov1}, we have
\begin{align*}
\begin{split}
\lim_{H \to 1}R_H(t,s)&=\frac{\sigma^2}{2}\left(-\theta(t-s)+\theta^2-\theta^2 e^{\frac{s-t}{\theta}}-\theta t e^{-\frac{s}{\theta}}+\theta(t-s)-\theta^2e^{-\frac{s}{\theta}}+\theta^2+\right.\\
&\left.-\theta t e^{-\frac{s}{\theta}}+\theta s e^{-\frac{t}{\theta}}+\theta^2 e^{-\frac{s}{\theta}}-\theta^2e^{-\frac{t}{\theta}}-\theta s e^{-\frac{t}{\theta}}-\theta^2e^{-\frac{t}{\theta}}+\theta^2 e^{\frac{s-t}{\theta}}+\right.\\
&\left.+2\theta t e^{-\frac{s}{\theta}}-2\theta^2 e^{-\frac{s}{\theta}}+2\theta^2 e^{-\frac{t+s}{\theta}}\right)\\
&=\sigma^2\theta^2(1-e^{-\frac{s}{\theta}})(1-e^{-\frac{t}{\theta}}),
\end{split}
\end{align*}
so we get equation \eqref{limcov}.\\
Furthermore, for $H=1$ and any continuous function $f$,  it holds that $$\int_0^tf(s)dB^1_s=\xi\int_0^tf(s)ds.$$ Hence, it is easy to see, substituting $f(s)=e^{\frac{s}{\theta}}$, that the process $V_t$ becomes
\begin{equation*}
V_t=\E[V_t]+\theta\sigma\xi(1-e^{-\frac{t}{\theta}})
\end{equation*}
and so the covariance is given by
\begin{equation*}
R_H(t,s)=\sigma^2\theta^2(1-e^{-\frac{s}{\theta}})(1-e^{-\frac{t}{\theta}})
\end{equation*}
that is the same value for the covariance as in the right-hand side of \eqref{limcov}.\\
Concerning $(iii)$, let us work with $t \ge s \ge 0$, since $R_H(t,s)$ is symmetric in $t$ and $s$. We have to distinguish three cases. First, if $s=0$, then $R_H(t,0) \equiv 0$ and also $\lim_{H \to 0}R_H(t,0)=0$, which is the first case of the right hand side of \eqref{limcov2}.\\
For the second case, if $t>s>0$, let us first rewrite $R_H(t,s)$ as
\begin{align}\label{eq:intpass}
\begin{split}
R_H(t,s)&=\frac{\sigma^2}{4}\left(-e^{\frac{s-t}{\theta}}\int_0^{t-s}e^{\frac{z}{\theta}}2Hz^{2H-1}dz+e^{\frac{t-s}{\theta}}\int_{t-s}^{t}e^{-\frac{z}{\theta}}2Hz^{2H-1}dz+\right.\\& \left.
-e^{-\frac{t+s}{\theta}}\int_s^{t}e^{\frac{z}{\theta}}2Hz^{2H-1}dz+e^{\frac{s-t}{\theta}}\int_0^{s}e^{-\frac{z}{\theta}}2Hz^{2H-1}dz+\right.\\ & \left.+2e^{-\frac{t+s}{\theta}}\int_0^{t}e^{\frac{z}{\theta}}2Hz^{2H-1}dz\right).
\end{split}
\end{align}
Now, recalling that
\begin{equation}\label{eq:auxpass}
\int_x^y 2H e^{az}z^{2H-1}dz=y^{2H}e^{ay}-x^{2H}e^{ax}-a\int_x^ye^{az}z^{2H}dz
\end{equation}
we have from equation \eqref{eq:intpass}
\begin{align}\label{eq:intpass2}
\begin{split}
R_H(t,s)&=\frac{\sigma^2}{4}\left(-(t-s)^{2H}+\frac{e^{\frac{s-t}{\theta}}}{\theta}\int_0^{t-s}e^{\frac{z}{\theta}}z^{2H}dz+t^{2H}e^{-\frac{s}{\theta}}-(t-s)^{2H}\right.\\
&\left.+\frac{e^{\frac{t-s}{\theta}}}{\theta}\int_{t-s}^{t}e^{-\frac{z}{\theta}}z^{2H}dz-t^{2H}e^{-\frac{s}{\theta}}+s^{2H}e^{-\frac{t}{\theta}}+\frac{e^{-\frac{t+s}{\theta}}}{\theta}\int_s^te^{\frac{z}{\theta}}z^{2H}dz\right.\\
&\left.+s^{2H}e^{-\frac{t}{\theta}}+\frac{e^{\frac{s-t}{\theta}}}{\theta}\int_0^s e^{-\frac{z}{\theta}}z^{2H}dz+2t^{2H}e^{-\frac{s}{\theta}}-2\frac{e^{-\frac{t+s}{\theta}}}{\theta}\int_0^te^{\frac{z}{\theta}}z^{2H}dz\right)\\
&=\frac{\sigma^2}{4}\left(-2(t-s)^{2H}+2t^{2H}e^{-\frac{s}{\theta}}+2s^{2H}e^{-\frac{t}{\theta}}+\frac{e^{\frac{s-t}{\theta}}}{\theta}\int_0^{t-s}e^{\frac{z}{\theta}}z^{2H}dz\right.\\
&\left.+\frac{e^{\frac{t-s}{\theta}}}{\theta}\int_{t-s}^{t}e^{-\frac{z}{\theta}}z^{2H}dz+\frac{e^{-\frac{t+s}{\theta}}}{\theta}\int_s^te^{\frac{z}{\theta}}z^{2H}dz\right.\\
&\left.+\frac{e^{\frac{s-t}{\theta}}}{\theta}\int_0^s e^{-\frac{z}{\theta}}z^{2H}dz-2\frac{e^{-\frac{t+s}{\theta}}}{\theta}\int_0^te^{\frac{z}{\theta}}z^{2H}dz\right).
\end{split}
\end{align}
Now we can use dominated convergence theorem in equation \eqref{eq:intpass2} to obtain
\begin{align*}
\begin{split}
\lim_{H \to 0}R_H(t,s)&=\frac{\sigma^2}{4}\left(-2+2e^{-\frac{s}{\theta}}+2e^{-\frac{t}{\theta}}+\frac{e^{\frac{s-t}{\theta}}}{\theta}\int_0^{t-s}e^{\frac{z}{\theta}}dz+\frac{e^{\frac{t-s}{\theta}}}{\theta}\int_{t-s}^{t}e^{-\frac{z}{\theta}}dz+\right.\\
&\left.\frac{e^{-\frac{t+s}{\theta}}}{\theta}\int_s^te^{\frac{z}{\theta}}dz+\frac{e^{\frac{s-t}{\theta}}}{\theta}\int_0^s e^{-\frac{z}{\theta}}dz-2\frac{e^{-\frac{t+s}{\theta}}}{\theta}\int_0^te^{\frac{z}{\theta}}dz\right)=\frac{\sigma^2}{2}e^{-\frac{t+s}{\theta}}
\end{split}
\end{align*}
which is the second case of the right hand side of \eqref{limcov2}.\\
For the third case, let us consider $t=s>0$. Then we have
\begin{align}\label{eq:intpassts}
\begin{split}
R_H(t,t)&=\frac{\sigma^2}{2}\left(\int_{0}^{t}e^{-\frac{z}{\theta}}2Hz^{2H-1}dz
+e^{-\frac{2t}{\theta}}\int_0^{t}e^{\frac{z}{\theta}}2Hz^{2H-1}dz\right).
\end{split}
\end{align}
By using equation \eqref{eq:auxpass} we obtain from \eqref{eq:intpassts}
\begin{align*}
\begin{split}
R_H(t,t)&=\frac{\sigma^2}{2}\left(2t^{2H}e^{-\frac{t}{\theta}}+\frac{1}{\theta}\int_0^t
e^{-\frac{z}{\theta}}z^{2H}dz-\frac{e^{-\frac{2t}{\theta}}}{\theta}\int_0^te^{\frac{z}{\theta}}z^{2H}dz\right),
\end{split}
\end{align*}
thus we can now use dominated convergence theorem to obtain
\begin{align*}
\begin{split}
\lim_{H \to 0} R_H(t,t)&=\frac{\sigma^2}{2}\left(2e^{-\frac{t}{\theta}}+\frac{1}{\theta}\int_0^te^{-\frac{z}{\theta}}dz-\frac{e^{-\frac{2t}{\theta}}}{\theta}\int_0^te^{\frac{z}{\theta}}dz\right)\\
&=\frac{\sigma^2}{2}\left(1+e^{-\frac{2t}{\theta}}\right)
\end{split}
\end{align*}
which is the third case of the right hand side of \eqref{limcov2}.\\
Finally, for $H=0$, using equation \eqref{eq:oterpass}, we have
\begin{equation*}
V_t=\E[V_t]+\sigma B_t^0-e^{-\frac{t}{\theta}}\frac{\sigma}{\theta}\int_0^te^{\frac{z}{\theta}}B_z^0dz.
\end{equation*}
It is obvious, since $V_0$ is deterministic, that if $t \wedge s=0$, we have $\Cov(V_t,V_s)=0$ and then we have the first case of the right hand side of \eqref{limcov2}.\\
Suppose that $t \wedge s>0$ and observe that the covariance of $V_t$ is given by
\begin{align}\label{eq:passCovH0}
\begin{split}
R_H(t,s)&=\sigma^2 \E[B_t^0B_s^0]-e^{-\frac{t}{\theta}}\frac{\sigma^2}{\theta}\int_0^te^{\frac{z}{\theta}}\E[B_z^0B_t^0]dz\\
&-e^{-\frac{s}{\theta}}\frac{\sigma^2}{\theta}\int_0^se^{\frac{z}{\theta}}\E[B_z^0B_s^0]dz+e^{-\frac{t+s}{\theta}}\frac{\sigma^2}{\theta^2}\int_0^t\int_0^se^{\frac{u+v}{\theta}}\E[B_u^0B_v^0]dudv.
\end{split}
\end{align}
Recall now that (see \cite[Lemma $4.1$]{novimish})
\begin{equation*}
\E[B_u^0B_v^0]=\begin{cases} 0 & t \wedge s=0, \\
\frac{1}{2} & t \not = s \mbox{ and }t,s>0, \\
1 & t=s>0.
\end{cases}
\end{equation*}
Let us first consider the case in which $t \not = s$. Then equation \eqref{eq:passCovH0} becomes
\begin{align*}
\begin{split}
R_H(t,s)&=\frac{\sigma^2}{2} -e^{-\frac{t}{\theta}}\frac{\sigma^2}{2\theta}\int_0^te^{\frac{z}{\theta}}dz\\
&-e^{-\frac{s}{\theta}}\frac{\sigma^2}{2\theta}\int_0^se^{\frac{z}{\theta}}dz+e^{-\frac{t+s}{\theta}}\frac{\sigma^2}{2\theta^2}\int_0^t\int_0^se^{\frac{u+v}{\theta}}dudv\\
&=\frac{\sigma^2}{2}-\frac{\sigma^2}{2}+\frac{\sigma^2}{2}e^{-\frac{t}{\theta}}-\frac{\sigma^2}{2}+\frac{\sigma^2}{2}e^{-\frac{s}{\theta}}\\
&+\frac{\sigma^2}{2}-\frac{\sigma^2}{2}e^{-\frac{t}{\theta}}-\frac{\sigma^2}{2}e^{-\frac{s}{\theta}}+\frac{\sigma^2}{2}e^{-\frac{t+s}{2}}=\frac{\sigma^2}{2}e^{-\frac{t+s}{2}}
\end{split}
\end{align*}
which is the second case of the right hand side of \eqref{limcov2}.\\
Now let us suppose that $t=s$. In this case
\begin{align*}
\begin{split}
R_H(t,t)&=\sigma^2-e^{-\frac{t}{\theta}}\frac{\sigma^2}{\theta}\int_0^te^{\frac{z}{\theta}}dz+e^{-\frac{2t}{\theta}}\frac{\sigma^2}{2\theta^2}\left(\int_0^te^{\frac{u}{\theta}}du\right)^2\\
&=\sigma^2-\sigma^2+\sigma^2e^{-\frac{t}{\theta}}+\frac{\sigma^2}{2}-\sigma^2e^{-\frac{t}{\theta}}+\frac{\sigma^2}{2}e^{-\frac{2t}{\theta}}\\
&=\frac{\sigma^2}{2}\left(1+e^{-\frac{2t}{\theta}}\right),
\end{split}
\end{align*}
that is the third case of the right hand side of \eqref{limcov2}.
\qed
\end{proof}
The behavior described in this Lemma is shown in Figure \ref{FigLemma}.
\begin{figure}[t]
\centering
{\includegraphics[width=0.49\textwidth]{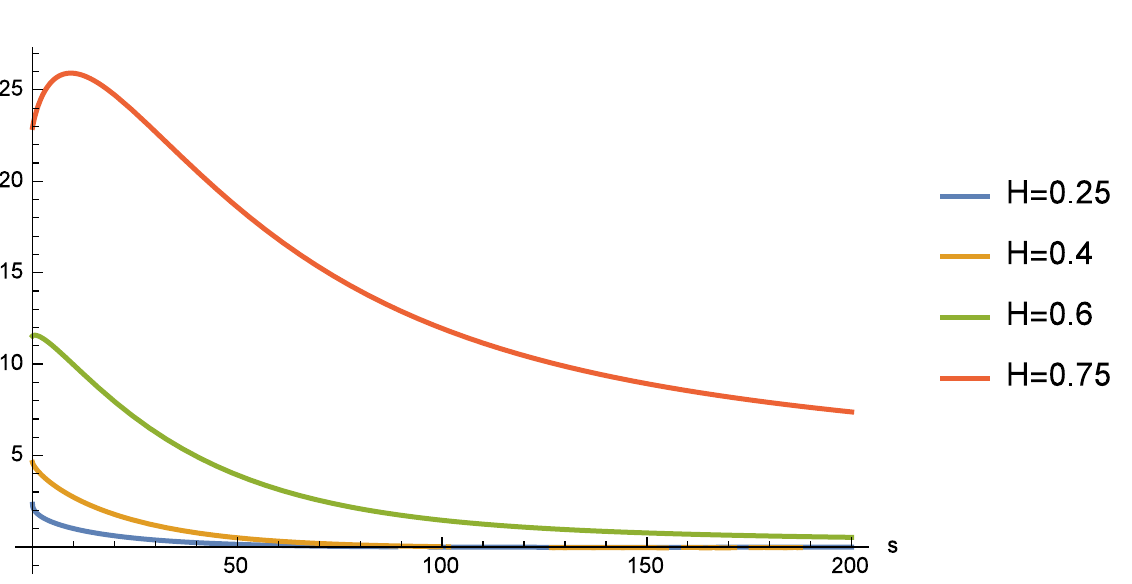}}
{\includegraphics[width=0.49\textwidth]{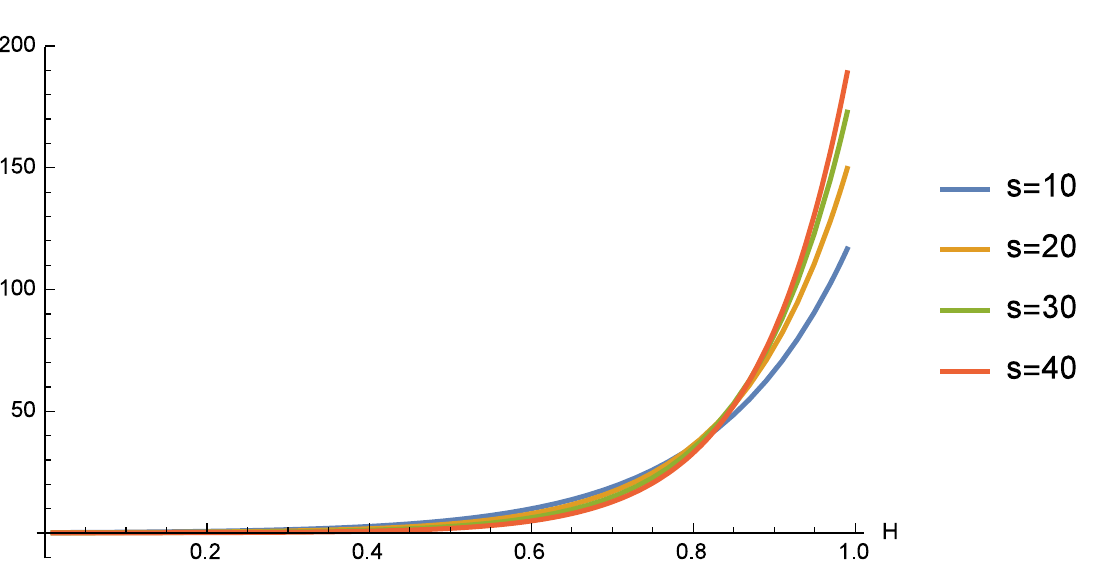}}
\caption{The covariance function $R_H(t,t+s)$ for $\theta=30$, $\sigma=1$, $t=10$ for a non-random $I$: on the left as a function of $s$ for different values of $H$; on the right as a function of $H$ for different values of $s$.}
\label{CovSigmabasso}
\end{figure}
\begin{figure}[t]
\centering
{\includegraphics[width=0.49\textwidth]{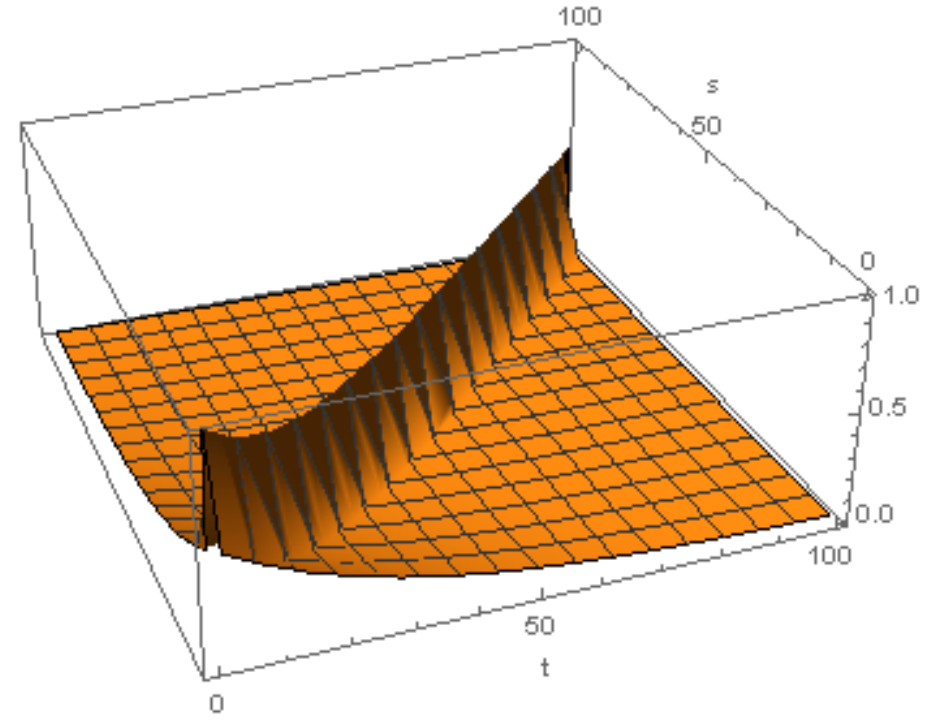}}
{\includegraphics[width=0.49\textwidth]{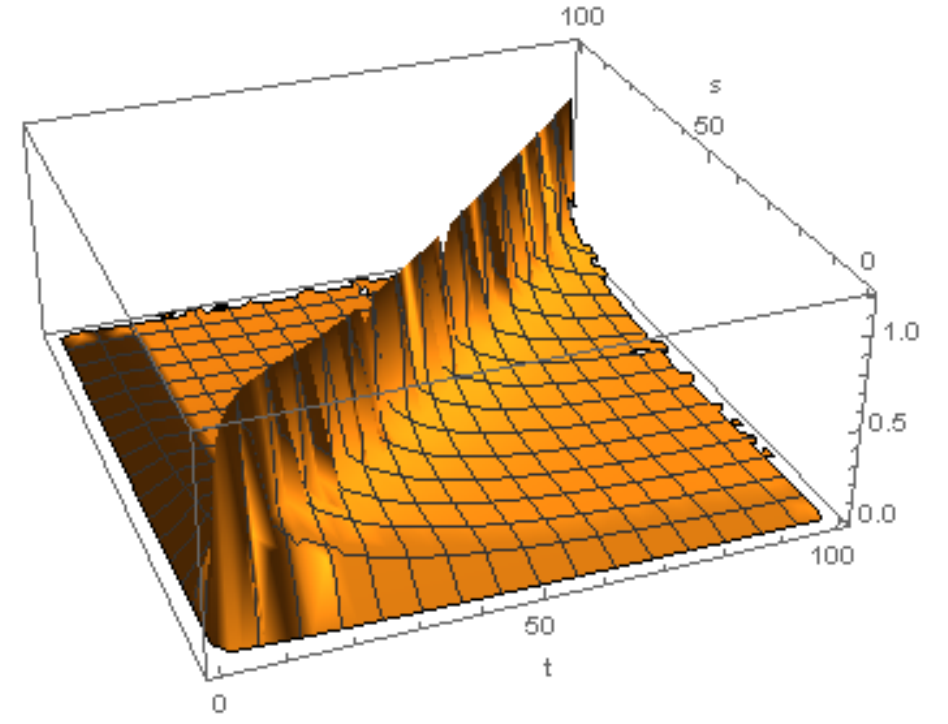}}
{\includegraphics[width=0.49\textwidth]{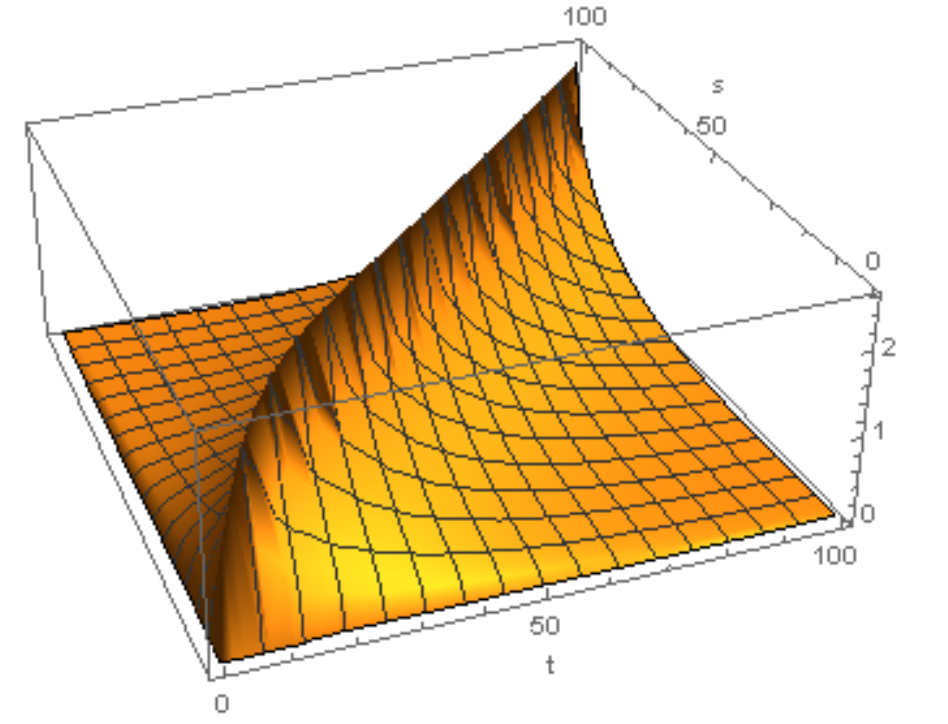}}
{\includegraphics[width=0.49\textwidth]{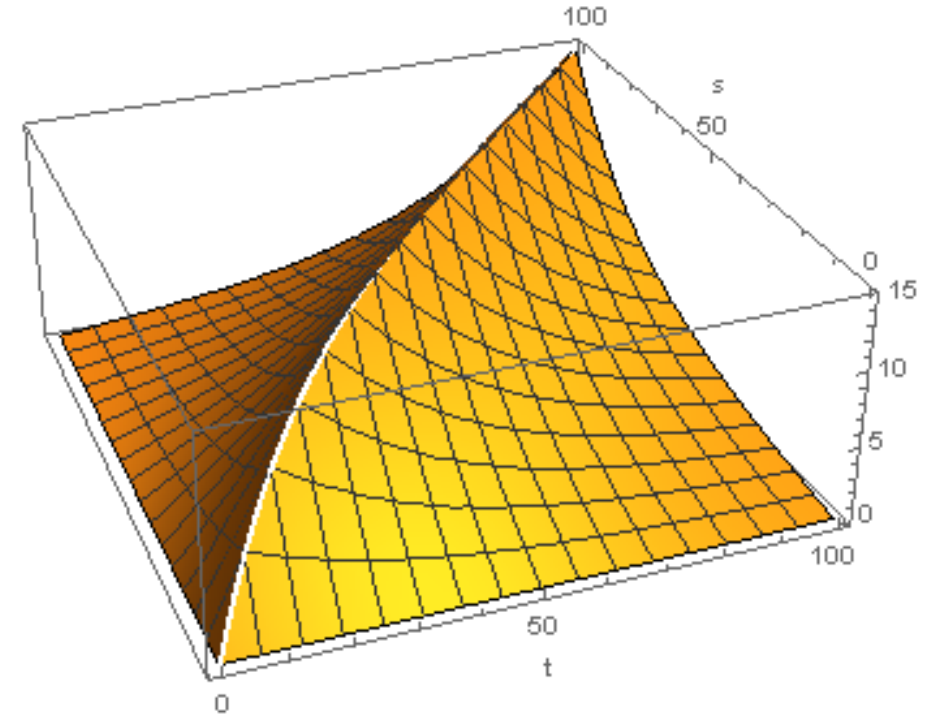}}
{\includegraphics[width=0.49\textwidth]{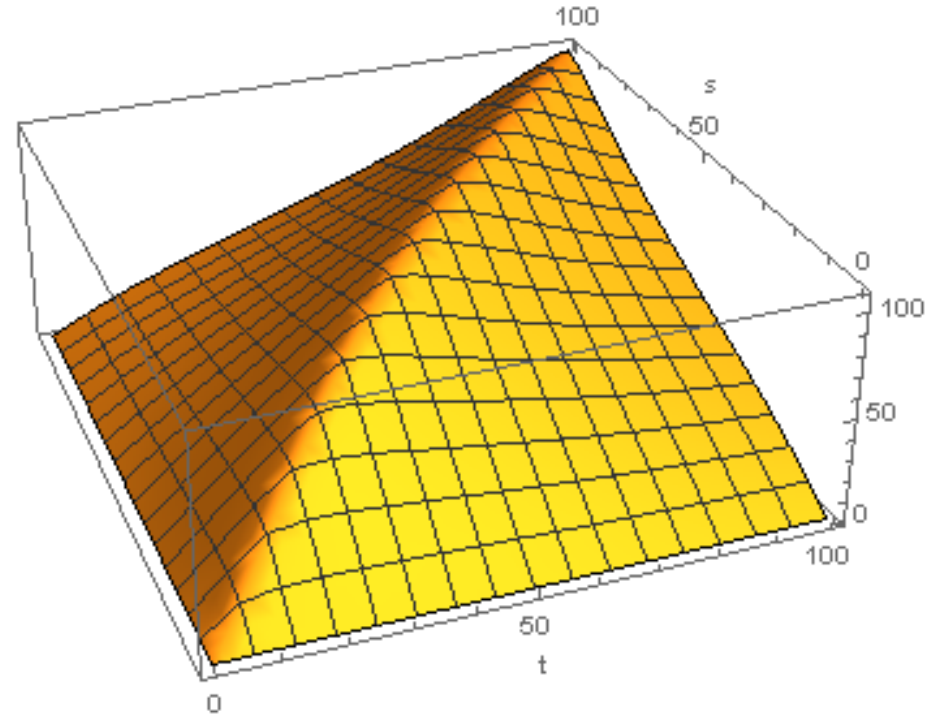}}
{\includegraphics[width=0.49\textwidth]{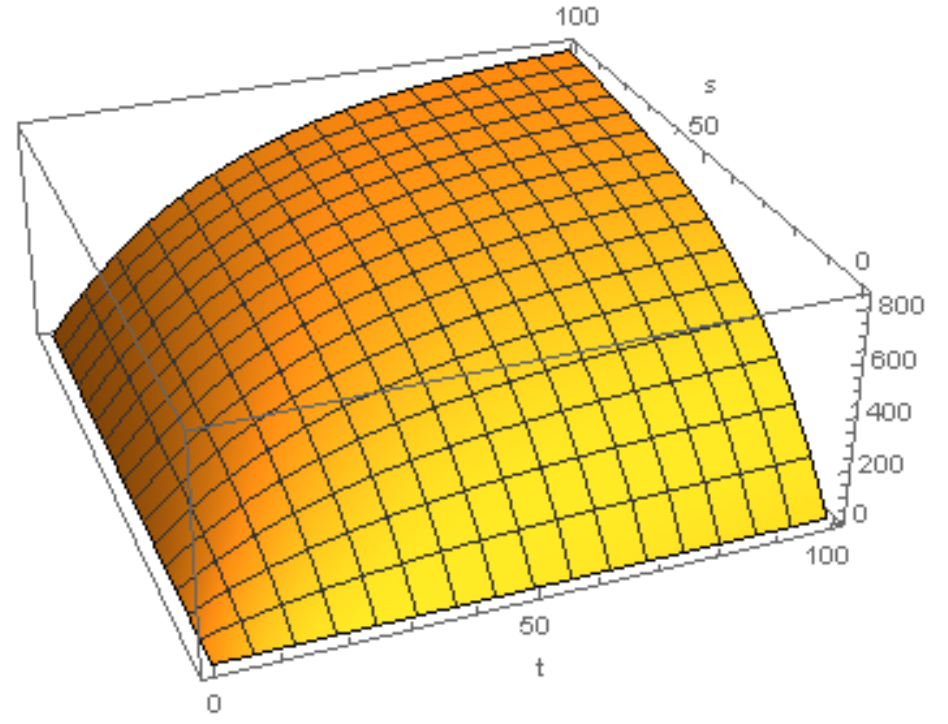}}
\caption{Plot of the function $R_H(t,s)$ for different values of $H$. Starting from the one in the upper left: $H=0$, $H=0.1$, $H=0.25$, $H=0.5$, $H=0.75$, $H=1$. Here we set $\theta=30$ and $\sigma=1$.}
\label{FigLemma}
\end{figure}

\subsubsection{Asymptotic behavior of variance. Non-random forcing term}\label{subsec-cov-2}
Let $t=s$.  Due to   simpler representations, we can better specify the asymptotic behavior of variance function as $t\to\infty$, in comparison with covariance.
\begin{lemma}\label{lem:asyVar}
Let $I$  be a non-random function. Then
\begin{equation}\label{varvar}
\lim_{t \to +\infty}\Var[V_t]= \sigma^2\theta^{2H}H\Gamma(2H).
\end{equation}
\end{lemma}
\begin{proof} Equality \eqref{eq:detvar} can be rewritten as
\begin{equation*}
\Var[V_t]=(1+e^{-\frac{2t}{\theta}})\sigma^2\theta^{2}C_H\int_{\mathrm{R}}\frac{|y|^{1-2H}}{1+\theta^2y^2}dy-2e^{-\frac{t}{\theta}}\rho(t)
\end{equation*}
where $\rho(t)$ has been defined by  equality  \eqref{eq:sfOUCov}. With the help of \cite{JeffreyZwillinger2007} (Formula 3.241 number 2), that is, $$\int_0^\infty\frac{x^{\alpha-1}}{1+x^\beta}dx=\frac{\pi}{\beta\sin(\frac{\pi\alpha}{\beta})},$$
 we are able to write this variance as
\begin{equation*}
\Var[V_t]=(1+e^{-\frac{2t}{\theta}})\frac{\sigma^2\theta^{2H}\Gamma(2H+1) }{2}-2e^{-\frac{t}{\theta}}\rho(t).
\end{equation*}
According to \eqref{covcov}, $\rho(t)\sim Kt^{2H-2}\to 0$  as $t \to +\infty$, and the equality \eqref{varvar} immediately follows.
\qed
\end{proof}
\begin{remark}The same result  can be obtained if $I$ is a stochastic process defined on $(\Omega', \cF', \bP')$. For $I_t=0$
it was obtained in \cite{KukushMishuraRalchenko2017}  but with the help of the Wiener integral representation. The proof presented now is much more elegant.
\end{remark}
From now on, we will denote $Var_H(t)=\Var[V_t]$.
\begin{figure}[t]
\centering
{\includegraphics[width=0.49\textwidth]{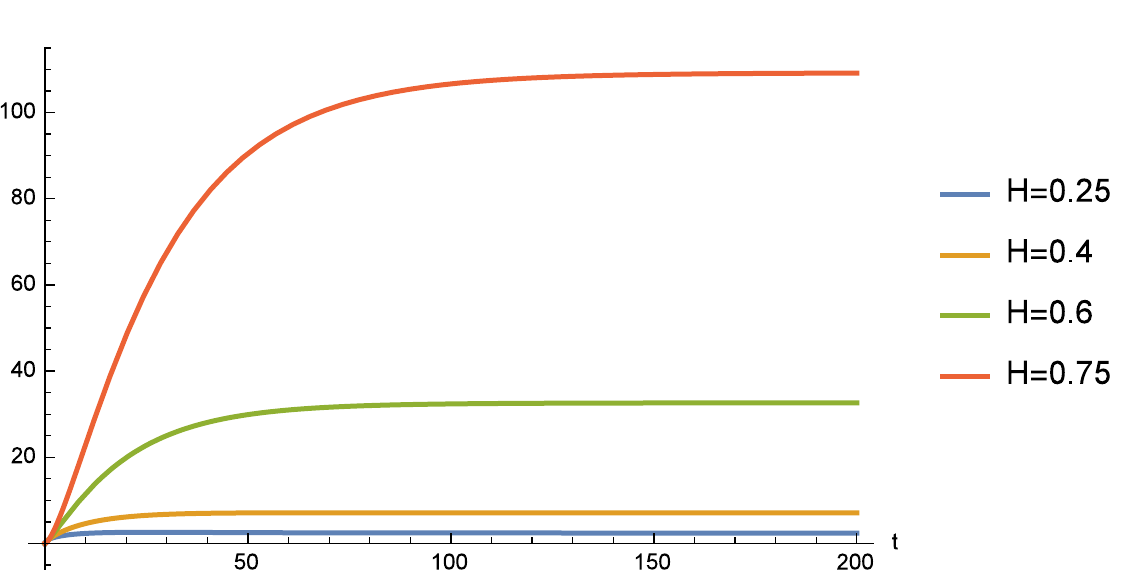}}
{\includegraphics[width=0.49\textwidth]{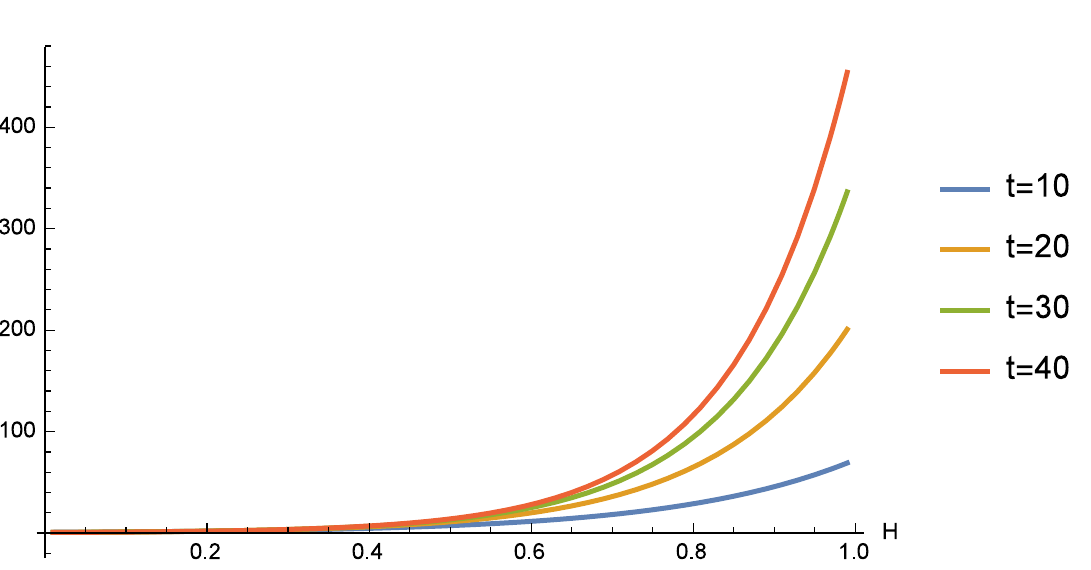}}
\caption{On the left: the function $Var_H(t)$ for a non-random $I$ with $\theta=30$, $\sigma=1$ and different values of $H$. On the right: $Var_H(t)$ for a non-random $I$ as a function of $H$ for $\theta=30$, $\sigma = 1$ and different values of $t$.}
\label{Var}
\end{figure}
For a non-random $I$, the variance $Var_H(t)$ has got different horizontal asymptote depending on the value of the Hurst parameter $H$, as shown in Lemma \ref{lem:asyVar}. This asymptotic behavior can be seen in Figure \ref{Var} on left.
Concerning the behavior of $Var_H(t)$ with respect to $H$, we can show the following corollary of Theorem \ref{lem:CovH}.
\begin{corollary}
Let a forcing term $I$ and an initial value be non-random. Then
\begin{itemize}
\item[$(i)$] The function $Var_H(t)=R_H(t,t)$ as a function of $H \in (0,1)$ and $t \in \mathrm{R}_+$ is continuous on $(0,1)\times \mathrm{R}_+$.
\item[$(ii)$] For any $t \in \mathrm{R}_+$
\begin{equation*}
\lim_{H \to 1}Var_H(t)=\sigma^2\theta^2(1-e^{-\frac{t}{\theta}})^2.
\end{equation*}
This result coincides with the formula that can be obtained if we directly substitute $B^1_t=t\xi$ into equality \eqref{eq:proc}. So, if we put
\begin{equation*}
Var_1(t)=\sigma^2\theta^2(1-e^{-\frac{t}{\theta}})^2
\end{equation*}
then $Var_H(t)$ becomes continuous on $(0,1]\times \mathrm{R}_+$.
\item[$(iii)$] For any $t \in \mathrm{R}_+$
\begin{equation*}
\lim_{H \to 0}Var_H(t)=\begin{cases} 0 & t=0 \\
\frac{\sigma^2}{2}\left(1+e^{-\frac{2t}{\theta}}\right) & t>0.
\end{cases}
\end{equation*}
This result coincides with the formula that can be obtained if we directly substitute $B^0_t=\eta_t$ into equality \eqref{eq:proc}. So, if we put
\begin{equation*}
Var_0(t)=\begin{cases} 0 & t=0 \\
\frac{\sigma^2}{2}\left(1+e^{-\frac{2t}{\theta}}\right) & t>0
\end{cases}
\end{equation*}
then $Var_H(t)$ is a continuous function on $[0,1]$ for any fixed $t$. However, as a function of the two variables $(H,t)$, it is discontinuous at point $H=0$ on the axis $t=0$.
\end{itemize}
\end{corollary}
This behavior with respect to the Hurst parameter is represented in Figure \ref{Var} on the right.
\subsubsection{On the asymptotic behavior of the covariance function. Random forcing term}\label{sub:asy}
Now, let $I$ be stochastic. Denote $C(t,s)=\Cov(V_t,V_s).$ In general, it is not obvious that $C(t,s)\to 0$ as $s \to +\infty$. The following proposition establishes some hypotheses under which $C(t,s)\to 0$ as $s \to +\infty$.
\begin{lemma}\label{prop:zerocov}
Suppose the forcing term $I$ verifies the hypotheses of Lemma \ref{prop:covgen}. Then $\lim_{s \to +\infty}C(t,s)=0$ under one of the following additional hypothesis:
\begin{itemize}
\item[$\mathcal{L}1$] The function $v \mapsto \int_0^{t}e^{\frac{u+v}{\theta}}c(u,v)du$ is in $L^1(\mathrm{R}^+)$;
\item[$\mathcal{L}2$] The following properties are verified:
\begin{itemize}
\item[$(i)$] $$\lim_{s \to +\infty}\int_0^{ s}\int_0^t e^{\frac{u+v}{\theta}}c(u,v)dudv=\infty;$$
\item[$(ii)$] $$\lim_{s\to +\infty}\int_0^t e^{\frac{u}{\theta}}c(u, s)du=0.$$
\end{itemize}
\end{itemize}
Moreover, property $(ii)$ of hypothesis $\mathcal{L}2$ is assured if the following properties hold:
\begin{itemize}
\item[$(iii)$] There exists $k \in L^1([0,t])$ such that $|c(u, s)|\le k(u)$ for almost all $u \in [0,t]$ and $s\ge 0$.
\item[$(iv)$] $\lim_{s \to +\infty}c(u, s)=0$.
\end{itemize}
\end{lemma}
\begin{proof}
Let us recall that now
\begin{equation*}
C(t,s)=R_H(t,s)+e^{-\frac{ t+s}{\theta}}\int_0^t\int_0^{ s}e^{\frac{u+v}{\theta}}c(u,v)dvdu,
\end{equation*}
and observe that
\begin{equation*}
\lim_{s \to +\infty}R_H(t,s)=0.
\end{equation*}
We want to evaluate the limit, applying Fubini theorem:
\begin{equation*}
\lim_{s \to +\infty}\frac{\int_0^t\int_0^{ s}e^{\frac{u+v}{\theta}}c(u,v)dvdu}{e^{\frac{ t+s}{\theta}}}=\lim_{s \to +\infty}\frac{\int_0^{s}\int_0^{t}e^{\frac{u+v}{\theta}}c(u,v)dudv}{e^{\frac{t+s}{\theta}}}.
\end{equation*}

Suppose we are under hypothesis $\mathcal{L}1$. Then
\begin{equation*}
\lim_{s \to +\infty}\int_0^{ s}\int_0^{t}e^{\frac{u+v}{\theta}}c(u,v)dudv=\int_0^{+\infty}\int_0^te^{\frac{u+v}{\theta}}c(u,v)dudv<+\infty
\end{equation*}
and
\begin{equation*}
\lim_{s \to +\infty}\frac{\int_0^{ s}\int_0^{t}e^{\frac{u+v}{\theta}}c(u,v)dudv}{e^{\frac{ t+s}{\theta}}}=0.
\end{equation*}
 
Suppose now we are under hypothesis $\mathcal{L}2$. Then  we can use L'Hospital's rule and get   
\begin{equation*}
\lim_{s \to +\infty}e^{-\frac{ t+s}{\theta}}\int_0^t\int_0^{ s}e^{\frac{u+v}{\theta}}c(u,v)dvdu=\lim_{s \to +\infty}\frac{\int_0^{t}e^{\frac{u}{\theta}}c(u, s)dudv}{\frac{1}{\theta}e^{\frac{t}{\theta}}}=0.
\end{equation*}
 
Finally, let us observe that properties $(iii)$ and $(iv)$ allow us to use dominated convergence theorem, because 
\begin{equation*}
e^{\frac{u}{\theta}}|c(u,s)|\le e^{\frac{t}{\theta}}k(u),
\end{equation*}
whence
\begin{equation*}
\lim_{s \to +\infty}\int_0^te^{\frac{u}{\theta}}c(u, s)du=\int_0^t\lim_{s \to +\infty}e^{\frac{u}{\theta}}c(u, s)du=0,
\end{equation*}
that is property $(ii)$.
\qed
\end{proof}
We have already shown in Remark \ref{rmk:cov} that if $I$ is non-random, then $C(t,s) \sim Ks^{2H-2}$ for some constant $K$. The following proposition provides some hypothesis under which $C(t,s)\sim Ks^{2H-2}$ even if $I$ is a stochastic process.
\begin{theorem}\label{prop:asymp}
Suppose the forcing term $I$ verifies the hypotheses of Lemma \ref{prop:covgen}. Then, $C(t,s) \sim Ks^{2H-2}$ for some constant $K$ under one of the following additional hypothesis: hypothesis $\mathcal{L}1$ from Lemma \ref{prop:zerocov} or
\begin{itemize}
 \item[$\mathcal{L}2'$] The following properties are verified:
\begin{itemize}
\item[$(i)$] $$\lim_{s \to +\infty}\int_0^{ s}\int_0^t e^{\frac{u+v}{\theta}}c(u,v)dudv=\infty;$$
\item[$(ii)'$] $$\lim_{s\to +\infty}s^{2-2H}\int_0^t e^{\frac{u}{\theta}}c(u, s)du=0.$$
\end{itemize}
\end{itemize}
Moreover property $(ii)'$ of hypothesis $\mathcal{L}2'$ is assured by the following properties:
\begin{itemize}
\item[$(iii)'$ ] there exists $k \in L^1([0,t])$ such that $s^{2-2H}|c(u, s)|\le k(u)$ for almost all $u \in [0,t]$ and $s\ge 0$;
\item[$(iv)'$ ] $\lim_{s \to +\infty}s^{2-2H}c(u, s)=0$.
\end{itemize}
\end{theorem}
\begin{proof}
Recalling that $R_H(t,s)\sim K s^{2H-2}$, let us evaluate, using Fubini's theorem as before,
\begin{equation}\label{eq:limit2}
\lim_{s \to +\infty}\frac{s^{2H-2}e^{\frac{ t+s}{\theta}}}{\int_0^{ s}\int_0^te^{\frac{u+v}{\theta}}c(u,v)dudv}.
\end{equation}
If we are under hypothesis $\mathcal{L}1$, we have
\begin{equation*}
\lim_{s \to +\infty}\frac{s^{2H-2}e^{\frac{ t+s}{\theta}}}{\int_0^{ s}\int_0^te^{\frac{u+v}{\theta}}c(u,v)dudv}=\infty.
\end{equation*}
Suppose then we are under hypothesis $\mathcal{L}2'$. Then    we can use L'Hospital's rule in equation \eqref{eq:limit2} to obtain
\begin{equation*}
\lim_{s \to +\infty}\frac{(2H-2)s^{2H-3}e^{\frac{ t+s}{\theta}}+\frac{1}{\theta}s^{2H-2}e^{\frac{ t+s}{\theta}}}{\int_0^te^{\frac{u +s}{\theta}}c(u, s)du}=\lim_{s \to +\infty}\frac{(2H-2)s^{2H-3}e^{\frac{t}{\theta}}+\frac{1}{\theta}s^{2H-2}e^{\frac{t}{\theta}}}{\int_0^te^{\frac{u}{\theta}}c(u, s)du}.
\end{equation*}
Analyzing the numerator of the right-hand side of the previous equality, we have  as $s \to +\infty$: 
\begin{equation*}
(2H-2)s^{2H-3}e^{\frac{t}{\theta}}+\frac{1}{\theta}s^{2H-2}e^{\frac{t}{\theta}}=\frac{(2H-2)+\frac{1}{\theta}s}{s^{3-2H}}e^{\frac{t}{\theta}}\sim \frac{e^{\frac{t}{\theta}}}{\theta}s^{2H-2}
\end{equation*}
then, by using hypothesis $(ii)'$, we have
\begin{equation*}
\lim_{s \to +\infty}\frac{(2H-2)s^{2H-3}e^{\frac{t}{\theta}}+\frac{1}{\theta}s^{2H-2}e^{\frac{t}{\theta}}}{\int_0^te^{\frac{u}{\theta}}c(u,t+s)du}=\lim_{s \to +\infty}\frac{\frac{e^{\frac{t}{\theta}}}{\theta}s^{2H-2}}{\int_0^te^{\frac{u}{\theta}}c(u,t+s)du}=\infty.
\end{equation*}
Finally, let us observe, as we provided in Lemma \ref{prop:zerocov}, that $(iii)'$ and $(iv)'$ allow us to apply the dominated convergence theorem to the integral in $(ii)'$.
\qed
\end{proof}
By this theorem, we know that under hypothesis $\mathcal{L}1$ or $\mathcal{L}2'$, $C(t,s)\sim K s^{2H-2}$. Hence, in such case $V_t$ exhibits a time non-homogeneous long-range dependence for $H>\frac{1}{2}$ and a time non-homogeneous short-range dependence for $H<\frac{1}{2}$. In particular, this behavior is induced only by the noise. For $H>\frac{1}{2}$, if we replace hypothesis $(ii)'$ with
\begin{itemize}
\item[$(ii)''$ ] \begin{equation*}
\lim_{s \to +\infty}s^{2-2H}\int_0^{t}e^{\frac{u}{\theta}}c(u, s)du=+\infty
\end{equation*}
\end{itemize}
we have again a time non-homogeneous long-range dependence, but this time this behavior is induced by the covariance of the process $I$. For $H<\frac{1}{2}$ we cannot conclude the same assertion. Indeed if we are under the hypotheses of proposition \ref{prop:zerocov} and $(ii)''$ is valid for some forcing term $I$, then we know that the covariance is such that $C(t,t+s)=O(s^{2H-2})$, which, for $H>\frac{1}{2}$, gives us the time non-homogeneous long-range dependence, since $2H-2>-1$; however, if $H<\frac{1}{2}$ (and then $2H-2<-1$) the fact that $C(t,t+s)=O(s^{2H-2})$ do not give us any additional information on the long-range and the short-range behavior.\\
The fact that, for $H>\frac{1}{2}$, $V$ preserves its correlation for long times is a good tool to be used in the field of neuronal modeling when one wants to include memory effects.
\section{A neuronal model}\label{sec-3}
\subsection{The model}
Let us show how one can use these results in the context of neuronal modeling. Consider a single neuron and denote with $V_t$ its membrane potential at the time $t\ge 0$, $\theta$ its characteristic time constant and $\widetilde{V}$ its resting potential. Denote the input stimulus as $I$ and suppose $\xi$ is a square integrable variable representing the initial value of the membrane potential. We consider a fractional LIF model supported by the following equation
\begin{equation*}
dV_t=\left[-\frac{1}{\theta}(V_t-\widetilde{V})+I_t\right]dt+\sigma dB^H_t, \ V_0=\xi
\end{equation*}
that is equation \eqref{eq:fdiffeq}, so that $V_t$ is a ffOU with initial data $\xi$ and forcing term $I$. We can suppose the initial membrane potential is given by $\xi=\widetilde{V}$, that is the asymptotic mean value of the potential when the neuron is not subject to any stimuli. The stimulus has to be chosen depending on what we want to model. It can be a constant stimulus $I_t=I_0$, an exponentially decaying stimulus $I_t=I_0e^{-\frac{1}{\tau}t}$ or also, if we want, for instance, to model the heartbeat or the rhythm of breath, a periodic stimulus.\\
An interesting case is the one in which the external input is given in a random time $T$ with given distribution. Indeed, one can consider the function
\begin{equation*}
h(t)=\begin{cases} 0 & t <0 \\
1 & t \ge 1
\end{cases}
\end{equation*}
and then define $I_t$
\begin{equation}\label{eq:stoccurrent}
I_t=I_0h(t-T).
\end{equation}
However, a more realistic model should admit a linear combination of these stimuli. Indeed, a neuron could receive different stimuli at different random times. We consider $n$ stimuli $I_t^i$ modeled as in equation \eqref{eq:stoccurrent}, for different random times $T_i$ and constant values $I_0^i$, in such a way that the total stimulus is the following stochastic process:
\begin{equation}\label{eq:stoccurrentgen}
I_t=\sum_{i=1}^{n}I_t^i=\sum_{i=1}^{n}I_0^ih(t-T_i), \ t\ge 0.
\end{equation}\\
Indeed, $T_i$ describes the activation time of a sodium or potassium channel of the neuron itself, depending on the membrane potential. In such case, the stimulus $I$ is a sort of auto-regulation stimulus and then the random times $T_i$ can defined on the same probability space of $V_t$. Moreover, $T_i$ can also be the firing time of other neurons. In such case, we have to define $T_i$ on different probability spaces. This stimulus allows one to combine two or more neurons in a more complex net. In the classical case, neurons coupling has been done in \cite{BuonocoreCaputoCarforaPirozzi2014, CarforaPirozzi2017}. In any case, remaining in the same probability space, it is also possible to model times of reaction of the neuron to an eventual stimulus by means of $T_i$.\\
From a mathematical point of view, equation \eqref{eq:covV} allows us to see the rule played by the covariance of the stochastic stimulus $I$ involved in the covariance of the process $V_t$. In particular, the more or less long-ranged or short-ranged memory of $V_t$ depends jointly on the values of the Hurst index $H\in (0,1)$ and the correlation function of the applied stimulus.
\subsection{The mean value function}
Consider the stochastic forcing term is given in equation \eqref{eq:stoccurrentgen} and observe that the single variable $I^i_t=I_0^ih(t-T_i)$ can be also written as
\begin{equation*}
I^i_t=\begin{cases}
0 & t < T_i\\
I_0^i & t\ge T_i.
\end{cases}
\end{equation*}
Denote with $F_{T_i}(t)$ the distribution function of $T_i$, i.e. $F_{T_i}(t)=\bP(T_i \le t)$. If we fix $t\ge 0$, the random variable $I_t^i$ can be written as $I_t^i=I_0^iZ_t^i$ where $Z_t^i$ is a Bernoulli random variable of parameter $F_{T_i}(t)$. In fact we have
\begin{equation*}
\bP(Z_t^i=1)=\bP(I_t^i=I_0^i)=\bP(t \ge T_i)=F_{T_i}(t).
\end{equation*}
By using this observation, we can show the following proposition.
\begin{proposition}
The mean value of the stimuli is given by
\begin{equation}\label{eq:meancorr}
\E[I_t]=\sum_{i=1}^{n}I_0^iF_{T_i}(t)
\end{equation}
while the mean value of the membrane potential process $V_t$ is given by
\begin{equation}\label{eq:expstocterm}
\E[V_t]=\widetilde{V}+\sum_{i=1}^{n}I_0^ie^{-\frac{t}{\theta}}\int_0^te^{\frac{s}{\theta}}F_{T_i}(s)ds.
\end{equation}
\end{proposition}
Equation \eqref{eq:expstocterm} is obtained from \eqref{eq:meancorr} by using equation \eqref{eq:meangen}.
\begin{remark}
If the variables $T_i$ for $i\le n$ are defined on a probability space $(\Omega',\cF',\bP')$ which is different from $(\Omega, \cF, \bP)$, then $\E[V_t]$ is a stochastic process on $(\Omega', \cF', \bP')$. In particular one can locally consider $I$ as a simple function and then we obtain
\begin{equation}\label{eq:expstoctermot}
\E[V_t]=\widetilde{V}+\sum_{i=1}^{n}I_0^i\theta(1-e^{-\frac{t}{\theta}})h(t-T_i).
\end{equation}
For modeling purposes, one could consider also $T_i$ defined on $(\Omega_i,\cF_i,\bP_i)$ for any $i \le n$, where the spaces $(\Omega_i, \cF_i,\bP_i)$ are two by two different and all different from $(\Omega, \cF, \bP)$ (for instance, if we want to model the stimuli coming from other $n$ neurons). In such case, we also obtain equation \eqref{eq:expstoctermot}.
\end{remark}
\subsection{The function $c(t,s)$: the general case}
Now we are interested in determining the term $c(t,s)$ of equation \eqref{eq:covV}. In particular we have the following result.
\begin{proposition}
For $t>s\ge 0$ we have
\begin{align}\label{eq:covHevgen}
\begin{split}
c(t,s)&=\sum_{i=1}^{n}\sum_{\substack{j=1 \\ j \not = i}}^{n}I_0^iI_0^j(F_{T_i,T_j}(t,s)-F_{T_i}(t)F_{T_j}(s))\\&+\sum_{i=1}^{n}(I_0^i)^2F_{T_i}(t)(1-F_{T_i}(s)),
\end{split}
\end{align}
where $F_{T_i,T_j}(t,s)=\bP(T_i \le t, T_j \le s)$.
\end{proposition}
\begin{proof}
Let us notice that
\begin{align}\label{eq:passcov}
\begin{split}
c(t,s)=\Cov(I_t,I_s)=\sum_{i,j=1}^{n}\Cov(I_t^i,I_s^j)=\sum_{i=1}^{n}\sum_{\substack{j=1 \\ j \not = i}}^{n}\Cov(I_t^i,I_s^j)+\sum_{i=1}^{n}\Cov(I_t^i,I_s^i)
\end{split}
\end{align}
so we have to determine $\Cov(I_t^i,I_s^j)$ for $i \not = j$ and $\Cov(I_t^i,I_s^i)$. To this, let us recall that for any $i,j$ and $t,s>0$:
\begin{equation*}
\Cov(I_t^i,I_s^j)=\E[I_t^iI_s^j]-\E[I_t^i]\E[I_s^j]=\E[I_t^iI_s^j]-I_0^iI_0^jF_{T_i}(t)F_{T_j}(s).
\end{equation*}
Let us first consider $j \not = i$. We have
\begin{equation*}
I_t^iI_s^j=\begin{cases} 0 & t<T_i \mbox{ or } s<T_j \\
I_0^iI_0^j & t\ge T_i \mbox{ and } s \ge T_j
\end{cases}
\end{equation*}
so that
\begin{equation*}
\E[I_t^iI_s^j]=I_0^iI_0^jF_{T_i,T_j}(t,s)
\end{equation*}
and
\begin{equation}\label{eq:covTinonord}
\Cov(I_t^i,I_s^j)=I_0^iI_0^j[F_{T_i,T_j}(t,s)-F_{T_i}(t)F_{T_j}(s)].
\end{equation}
For $i=j$, let us remark that
\begin{equation*}
I_t^iI_s^i=\begin{cases} 0 & \min\{s,t\}<T_i \\
(I_0^i)^2 & \min\{s,t\}\ge T_i
\end{cases}
\end{equation*}
so we have
\begin{equation*}
\E[I_t^i,I_s^i]=(I_0^i)^2F_{T_i}(\min\{t,s\})
\end{equation*}
and
\begin{equation}\label{eq:cov3}
\Cov(I_t^i,I_s^i)=(I_0^i)^2F_{T_i}(\min\{t,s\})(1-F_{T_i}(\max\{t,s\})).
\end{equation}
Thus, supposing $t>s \ge 0$ we obtain equation \eqref{eq:covHevgen} from \eqref{eq:passcov} by using \eqref{eq:covTinonord} in the first summation and \eqref{eq:cov3} in the second one.
\qed
\end{proof}
If $I$ and $B_t^H$ are uncorrelated, the covariance and variance functions of $V_t$ are given by equations \eqref{eq:covV} and \eqref{eq:vargen} with $c(u,v)$ defined in \eqref{eq:covHevgen}.\\
For modelling purposes, one could be interested in two particular subcases: the one in which the activation times $T_i$ are independent and the one in which they are ordered.
\subsection{The function $c(t,s)$: the independent activation time case}
Suppose we want to model a neuron that is subject to the constant stimuli that other neurons (that are independent of each other) send to it after they spike for the first time (it is the case, for instance, of the retinal neurons, as described in \cite{Shepherd1998}). Thus we have to suppose that the firing times of the neurons are independent random variables $T_i$.\\ 
In this case, it is really easy to see the following corollary
\begin{corollary}
If the variable $T_i$ are independent, then
\begin{equation}\label{eq:covscor}
c(t,s)=\sum_{i=1}^{n}(I_0^i)^2F_{T_i}(t)(1-F_{T_i}(s)).
\end{equation}
\end{corollary}
\begin{proof}
One has just to notice that if the variables $T_i$ are independent, then $F_{T_i,T_j}(t,s)=F_{T_i}(t)F_{T_j}(s)$: using this observation in equation \eqref{eq:covHevgen}, we obtain equation \eqref{eq:covscor}.
\qed
\end{proof}
\subsection{The function $c(t,s)$: the ordered activation time case}
In order to model the auto-regulation stimuli, one can suppose the excitatory/inhibitory channels have a priority activation order and then the variable $T_1\le T_2 \le \dots \le T_n$ almost surely. In this case we can characterize further the function $c(t,s)$. Suppose the random variables $J_i=T_{i}-T_{i-1}$ for $i\ge 1$ (with $T_0=0$ almost surely) are independent from each other and (then) from $T_{i-1}$. Furthermore, assume that $T_i$ and $J_j$ are absolutely continuous variables with densities $f_{T_i}$ and $f_{J_j}$ and joint density $f_{T_i,J_j}$.\\
For this setting we can show the following Proposition.
\begin{proposition}
For $t \ge s$ we have
\begin{align}\label{eq:covord}
\begin{split}
c(t,s)&=\sum_{j=1}^{n}\sum_{j<i\le n}I_0^iI_0^jF_{T_j}(s)(1-F_{T_i}(t))\\&+\sum_{j=2}^{n}\sum_{i<j}I_0^iI_0^j\int_0^s f_{T_i}(u)\left(\int_u^t(f_{J_{i+1}}\ast \dots \ast f_{J_j})(v-u)dv-F_{T_j}(s)\right)du\\&+\sum_{j=1}^{n}(I_0^j)^2F_{T_j}(t)(1-F_{T_j}(s)).
\end{split}
\end{align}
\end{proposition}
\begin{proof}
Let us first observe that
\begin{align}\label{eq:passcov2}
\begin{split}
c(t,s)&=\sum_{i=1}^{n}\sum_{j=1}^{n}\Cov(I_t^i,I_s^j)\\&=\sum_{j=2}^{n}\sum_{i<j}\Cov(I_t^i,I_s^j)+\sum_{j=1}^{n}\sum_{j<i\le n}\Cov(I_t^i,I_s^j)+\sum_{j=1}^{n}\Cov(I_t^j,I_s^j)
\end{split}
\end{align}
so we have to determine $\Cov(I_t^i,I_s^j)$ for $j>i$, $j<i$ and $j=i$.\\
Suppose first $i<j$. If $t\ge s$ then, since $T_j\ge T_i$ almost surely, we have
\begin{equation*}
I_t^iI_s^j=\begin{cases} 0 & s<T_j \\
I_0^iI_0^j & s \ge T_j
\end{cases} \mbox{ and } \E[I_t^iI_s^j]=I_0^iI_0^jF_{T_j}(s)
\end{equation*}
so that
\begin{equation}\label{eq:cov1}
\Cov(I_t^i,I_s^j)=I_0^iI_0^jF_{T_j}(s)(1-F_{T_i}(t)).
\end{equation}
If we suppose $j>i$, we can obtain a different representation for $F_{T_i,T_j}(t,s)$. For $u<v$
\begin{multline*}
f_{T_i,T_j}(u,v)=f_{T_i,\sum_{k=i+1}^{j}J_k}(u,v-u)=f_{T_i}(u)f_{\sum_{k=i+1}^{j}J_k}(v-u)\\=f_{T_i}(u)(f_{J_{i+1}}\ast\dots \ast f_{J_j})(v-u)
\end{multline*}
while for $u>v$ $f_{T_i,T_j}(u,v)=0$. Thus we have for $t<s$
\begin{multline*}
\bP(t \ge T_i, s \ge T_j)=\int_0^t\int_u^sf_{T_i,T_j}(u,v)dvdu\\=\int_0^t\int_u^sf_{T_i}(u)(f_{J_{i+1}}\ast\dots \ast f_{J_j})(v-u)dvdu,
\end{multline*}
and, by also writing $F_{T_i}(t)=\int_0^tf_{T_i}(u)du$
\begin{equation}\label{eq:cov2}
\Cov(I_t^i,I_s^j)=I_0^iI_0^j\int_{0}^{t}f_{T_i}(u)\left(\int_{u}^{s}(f_{J_{i+1}}\ast\dots \ast f_{J_j})(v-u)dv-F_{T_j}(s)\right)du.
\end{equation}
For $j=i$, we already have the expressions for $\Cov(I_t^i,I_s^i)$.\\
Finally, we obtain equation \eqref{eq:covord} from \eqref{eq:passcov2} by using \eqref{eq:cov1} and \eqref{eq:cov2}
\qed
\end{proof}
\subsection{The single activation time case}
An other interesting case is given by posing $n=1$, and then $T_1=T$ and $I_0^1=I_0$. In particular we have
\begin{equation*}
\E[I_t]=I_0F_{T}(t)
\end{equation*}
and, for $t\ge s$
\begin{equation*}
\E[I_tI_s]=I_0^2F_{T}(s)
\end{equation*}
thus
\begin{equation*}
c(t,s)=\Cov(I_t,I_s)=I_0^2F_{T}(s)(1-F_{T}(t)).
\end{equation*}
In general we have, for any $(t,s)$
\begin{equation*}
c(t,s)=\Cov(I_t,I_s)=I_0^2F_{T}(\min\{t,s\})(1-F_{T}(\max\{t,s\})),
\end{equation*}
that agrees with equation \eqref{eq:covscor}.
Thus equation \eqref{eq:expstocterm} becomes
\begin{equation}\label{eq:meann1}
\E[V_t]=\widetilde{V}+I_0e^{-\frac{t}{\theta}}\int_0^te^{\frac{s}{\theta}}F_{T}(s)ds
\end{equation}
while equations \eqref{eq:covV} and \eqref{eq:vargen} become, for $t \ge s$,
\begin{align*}
\begin{split}
C(t,s)&=R_H(t,s)+I_0^2e^{-\frac{t+s}{\theta}}\left(\int_0^t\int_0^u e^{\frac{u+v}{\theta}}F_{T}(v)(1-F_{T}(u))dvdu\right.\\&\left.+\int_0^t\int_u^{s} e^{\frac{u+v}{\theta}}F_{T}(u)(1-F_{T}(v))dvdu\right)
\end{split}
\end{align*}
and
\begin{align*}
\begin{split}
Var(t)&=R_H(t,t)+I_0^2e^{-\frac{2t}{\theta}}\left(\int_0^t\int_0^ue^{\frac{u+v}{\theta}}F_{T}(v)(1-F_{T}(u))dvdu\right.\\&\left.+\int_0^t\int_u^te^{\frac{u+v}{\theta}}F_{T}(u)(1-F_{T}(v))dvdu\right).
\end{split}
\end{align*}
We show the plot of the function in equation \eqref{eq:meann1}, together with a simulated sample path, in Figure \ref{Simpath3}.\\
In this case one can show the following Proposition.
\begin{proposition}
We have $\lim_{s \to +\infty}C(t,t+s)=0$ independently from the choice of the distribution of $T$.
\end{proposition}
\begin{proof}
Since $F_T(t)\le 1$ for any $t>0$, we have that $c(u,t+s)\le I_0^2$. Moreover, for $u<t+s$, we have
\begin{equation*}
c(u,t+s)=I_0^2F_T(u)(1-F_T(t+s))
\end{equation*}
and then $\lim_{s \to +\infty}c(u,t+s)=0$. Thus, property ii of hypothesis $\mathcal{L}2$ in Proposition \ref{prop:zerocov} is verified. Moreover, since $c(t,s)\ge 0$, we have that one between hypothesis $\mathcal{L}1$ and property i of hypothesis $\mathcal{L}2$ have to be verified, so, without making any calculation, we can conclude that $\lim_{s \to +\infty}C(t,s)=0$.
\qed
\end{proof}
However long-range or short-range dependence depend on the choice of $F_T$. Indeed we can show the following two propositions.
\begin{proposition}\label{prop:exp}
If $T \sim Exp(\lambda)$ then $C(t,s)\sim Ks^{2H-2}$ for a constant $K>0$.
\end{proposition}
\begin{proof}
Let us recall that since $c(t,s)\ge 0$, we have that one between hypothesis $\mathcal{L}1$ and property i of hypothesis $\mathcal{L}2'$ have to be verified. So let us only verify property ii'.\\
Let us observe that since $F_T(t)=1-e^{-\lambda t}$, then, for fixed $t>0$ and $u \in [0,t]$, $s^{2-2H}c(u,t+s)$ is bounded uniformly with respect to $u$. Indeed, for fixed $t>0$, observe that
\begin{equation*}
s^{2-2H}c(u,t+s)=s^{2-2H}I_0^2(1-e^{-\lambda u})e^{-\lambda(t+s)} \le s^{2-2H}I_0^2e^{-\lambda(t+s)}=:I_0^2f(s).
\end{equation*}
We have that $f$ is a continuous function with $f(0)=0$, $f(s)>0$ for any $s>0$ and $\lim_{s \to +\infty}f(s)=0$, so there exists a constant $M>0$ such that $f(s)\le M$ for any $s \in \mathrm{R}_+$ and in particular
\begin{equation*}
s^{2-2H}c(u,t+s)\le I_0^2M
\end{equation*}
Moreover
\begin{equation*}
\lim_{s \to +\infty}s^{2-2H}c(u,t+s)=\lim_{s \to +\infty}s^{2-2H}I_0^2(1-e^{-\lambda u})e^{-\lambda(t+s)}=0,
\end{equation*}
thus property ii' of hypothesis $\mathcal{L}2'$ in Proposition \ref{prop:asymp} is verified. Hence we can conclude that $C(t,s)\sim Ks^{2H-2}$.
\qed
\end{proof}
\begin{proposition}\label{prop:stab}
Let $T$ be a non-negative $\alpha$-stable random variable:
\begin{itemize}
\item if $\alpha>2-2H$ then $C(t,t+s) \sim Ks^{2H-2}$ for some constant $K$ independent from the stimulus as $s \to +\infty$;
\item if $\alpha<2-2H$ then the asymptotic behaviour of $C(t,t+s)$ as $s \to +\infty$ depends on $\alpha$ and on the stimulus.
\end{itemize}
\end{proposition}
\begin{proof}
Let us recall that since $c(t,s)\ge 0$, we have that one between hypothesis $\mathcal{L}1$ and property i of hypothesis $\mathcal{L}2'$ have to be verified. So let us only verify property ii' or ii''.\\
Let us observe that $1-F_T(t) \sim \widetilde{K}t^{-\alpha}$ as $t \to +\infty$ for some $\widetilde{K}$. If we choice the index of stability $\alpha$ to be such that $\alpha>2-2H$, then $s^{2-2H}c(u,t+s)$ is bounded uniformly with respect to $u \in [0,t]$. Indeed, for $t>0$ fixed, observe that:
\begin{equation*}
s^{2-2H}c(u,t+s)=s^{2-2H}I_0^2F_T(u)(1-F_T(t+s))\le I_0^2s^{2-2H}(1-F_T(t+s))=:I_0^2f(s).
\end{equation*}
We have that $f$ is a continuous function with $f(0)=0$, $f(s)>0$ for any $s>0$ and
$$\lim_{s \to +\infty}f(s)=\lim_{s \to +\infty}\widetilde{K}s^{2-2H-\alpha}=0$$
since $\alpha>2-2H$. So there exists a constant $M>0$ such that $f(s)\le M$ for any $s \in \mathrm{R}_+$ and in particular
\begin{equation*}
s^{2-2H}c(u,t+s)\le MI_0^2.
\end{equation*}
Moreover
\begin{equation*}
\lim_{s \to +\infty}s^{2-2H}c(u,t+s)=\lim_{s \to +\infty}s^{2-2H}F_T(u)(1-F_T(t+s))=\lim_{s \to +\infty} \widetilde{K} s^{2-2H-\alpha}=0
\end{equation*}
wince $2-2H-\alpha<0$ and then property ii' is verified.\\
If $\alpha<2-2H$, then
\begin{align*}
\begin{split}
\lim_{s \to +\infty}s^{2-2H}\int_0^te^{\frac{u}{t}}c(u,t+s)du&=\lim_{s \to +\infty}s^{2-2H}I_0^2(1-F_T(t+s))\int_0^te^{\frac{u}{t}}F_T(u)\\
&=\lim_{s \to +\infty} \widetilde{K}s^{2-2H-\alpha}=+\infty
\end{split}
\end{align*}
since $2-2H-\alpha>0$, so property $(ii)''$ is verified.
\qed
\end{proof}
Let us also remark that in such case, for $\alpha=2-2H$, the asymptotic behavior of $C(t, s)$ as $s \to +\infty$ also depends on $H$ (since $\alpha$ depends on $H$), but we do not have in general $C(t,s)\sim Ks^{2-2H}$. Moreover, if we have in such case $C(t, s)\sim Ks^{2-2H}$, then the constant $K$ could depend on the stimulus.\\
These two choices of $T$ are not arbitrary. For instance, we could desire to model a couple of neuron of which the first one sends a constant signal after firing to the second one. In particular, we could choice two different models for these neurons, since the could have two specifically different functions. Thus, let us suppose the second neuron is described by the model we presented here. If we want to describe the first neuron with a classical LIF model, then the firing time can be approximated with an exponential random variable (see, for instance, \cite{BuonocoreCaputoPirozziRicciardi2011}), justifying the choice in proposition \ref{prop:exp}. However, it has been shown in \cite{MandelbrotGerstein1964} that a good choice for the distribution of a firing time could be the non-negative stable one, justifying then the distribution of $T$ in Proposition \ref{prop:stab}.\\
Another particular case is the one with constant stimulus, that can be obtained from this case by choosing $T \equiv 0$. In such case $F_T(t)=1$ for any $t\ge0$. Thus we have:
\begin{equation}\label{eq:meanconst}
\E[V_t]=\widetilde{V}+I_0\theta(1-e^{-\frac{t}{\theta}})
\end{equation}
while, since $c(t,s)=0$ for any $t,s>0$, $C(t,s)=R_H(t,s)$ and $Var(t)=R_H(t,t)$. We show the plot of the function in equation \eqref{eq:meanconst}, together with a simulated sample path, in Figure \ref{Simpath1}.\\
Finally, let us remark that, to obtain a realistic model, one has to estimate $H$, keeping in consideration how the variance and the covariance (and their limit values) of the process depend on $H$, and how these functions vary when we change $H$, as it has been studied in subsection \ref{sub:cov}.
\section{Simulation algorithms}\label{sec-4}
Finally we want to show some simulation algorithms for the process $V$. Algorithms for the simulation of the sample paths of $V$ can be widely used to numerically approximate its first passage time densities through a constant threshold $V_{th}$. First passage times of such processes through constant thresholds are very important in neuronal modeling: they represent the spiking time of a neuron and then, if the process is subject to a memory reset, also the inter-spike intervals of $V$. Due to our modeling interests, we will pose $\xi\equiv \widetilde{V}$ almost surely.\\
To simulate the trajectories of the process $V$, one can use Euler approximation method, obtaining the following recursive formula, based on equation \eqref{eq:fdiffeq} with initial data $V_0=\widetilde{V}$
\begin{equation*}
\begin{cases}
V_n=V_{n-1}+\left(-\frac{1}{\theta}(V_{n-1}-\widetilde{V})+I_{n-1}\right)\Delta t + \sigma (B_n^H-B_{n-1}^H)\\
V_0=\widetilde{V}
\end{cases}
\end{equation*}
where $t_n=n\Delta t$, $V_n=V_{t_n}$, $I_n=I(t_n)$ and $B^H_n=B^H_{t_n}$. Fractional Gaussian noise $G_n=B_n^H-B_{n-1}^H$ can be simulated using Circulant Embedding method, as done in \cite{PerrinHarbaJennaneIribarren2002}. Properties related to the convergence of Euler schemes for stationary solutions of SDEs are investigated in \cite{CohenPanloup2011}.\\

Here, we propose another simulation algorithm. By using integration by parts formula we have
\begin{equation*}
\int_{0}^{t}e^{\frac{s}{\theta}}dB^H_s=e^{\frac{t}{\theta}}B^H_t-\frac{1}{\theta}\int_{0}^{t}e^{\frac{s}{\theta}}B^H_sds
\end{equation*}
so that equations \eqref{eq:proc} and \eqref{eq:proc1} can be rewritten as
\begin{equation*}
V_t=\widetilde{V}+\sigma B^H_t+e^{-\frac{t}{\theta}}\left[\int_{0}^{t}I(s)e^{\frac{s}{\theta}}ds-\frac{\sigma}{\theta}\int_{0}^{t}e^{\frac{s}{\theta}}B^H_sds\right].
\end{equation*}
Choose a very small time interval $\Delta t$ and suppose we want to simulate our process in $[0,N\Delta t]$. We can simulate the increments of the process $B_t^H$ in $[0,N\Delta t]$ using Circulant Embedding method, so it is easy to obtain a simulation of $B_t^H$ in $[0,N\Delta t]$. Suppose we have already simulated $V_n$ and we want to simulate $V_{n+1}$. We have
\begin{align}
\begin{split}\label{eq:vnpform}
V_{n+1}&= \widetilde{V}+\sigma B^H_{n+1}+e^{-\frac{t_{n+1}}{\theta}}\left[\int_{0}^{t_{n+1}}I(s)e^{\frac{s}{\theta}}ds-\frac{\sigma}{\theta}\int_{0}^{t_{n+1}}e^{\frac{s}{\theta}}B^H_sds\right]\\&= \widetilde{V}+\sigma B^H_{n+1}+e^{-\frac{\Delta t}{\theta}}e^{-\frac{t_n}{\theta}}\left[\int_{0}^{t_{n}}I(s)e^{\frac{s}{\theta}}ds-\frac{\sigma}{\theta}\int_{0}^{t_{n}}e^{\frac{s}{\theta}}B^H_sds\right]\\&+ e^{-\frac{t_{n+1}}{\theta}}\left[\int_{t_n}^{t_{n+1}}I(s)e^{\frac{s}{\theta}}ds-\frac{\sigma}{\theta}\int_{t_n}^{t_{n+1}}e^{\frac{s}{\theta}}B^H_sds\right].
\end{split}
\end{align}
Note that
\begin{equation*}
V_n=\widetilde{V}+\sigma B^H_{n}+e^{-\frac{t_n}{\theta}}\left[\int_{0}^{t_{n}}I(s)e^{\frac{s}{\theta}}ds-\frac{\sigma }{\theta}\int_{0}^{t_{n}}e^{\frac{s}{\theta}}B^H_sds\right],
\end{equation*}
so
\begin{equation}\label{eq:vnform}
e^{-\frac{t_n}{\theta}}\left[\int_{0}^{t_{n}}I(s)e^{\frac{s}{\theta}}ds-\frac{\sigma }{\theta}\int_{0}^{t_{n}}e^{\frac{s}{\theta}}B^H_sds\right]=V_n-\widetilde{V}-\sigma B^H_n.
\end{equation}
Applying \eqref{eq:vnform} to \eqref{eq:vnpform}, we obtain
\begin{align*}
V_{n+1}&=\widetilde{V}(1-e^{-\frac{\Delta t}{\theta}})+e^{-\frac{\Delta t}{\theta}}V_n+\sigma B^H_{n+1}-\sigma e^{-\frac{\Delta t}{\theta}}B^H_n\\&+e^{-\frac{t_{n+1}}{\theta}}\left[\int_{t_n}^{t_{n+1}}I(s)e^{\frac{s}{\theta}}ds-\frac{\sigma }{\theta}\int_{t_n}^{t_{n+1}}e^{\frac{s}{\theta}}B_s^Hds\right].
\end{align*}
If $\Delta t$ is sufficiently small, we can approximate the remaining integrals with a single step of a closed Newton-Cotes formula. For instance, if we use the trapezoid rule, we have
\begin{align*}
V_{n+1}&=\widetilde{V}(1-e^{-\frac{\Delta t}{\theta}})+e^{-\frac{\Delta t}{\theta}}V_n+\sigma\left(1-\frac{\Delta t}{2\theta}\right)B^H_{n+1}\\&-\sigma e^{-\frac{\Delta t}{\theta}}\left(1+\frac{\Delta t}{2\theta}\right)B^H_n+\frac{\Delta t}{2}(I_{n+1}+e^{-\frac{\Delta t}{\theta}}I_n).
\end{align*}
and finally we obtain a recursive formula
\begin{equation*}
\begin{cases}
\begin{split}
V_{n+1}=\widetilde{V}(1-e^{-\frac{\Delta t}{\theta}})+e^{-\frac{\Delta t}{\theta}}V_n+\sigma\left(1-\frac{\Delta t}{2\theta}\right)B^H_{n+1}+\\-\sigma e^{-\frac{\Delta t}{\theta}}\left(1+\frac{\Delta t}{2\theta}\right)B^H_n+\frac{\Delta t}{2}(I_{n+1}+e^{-\frac{\Delta t}{\theta}}I_n)
\end{split}\\
V_0=\widetilde{V}.
\end{cases}
\end{equation*}
In order to simulate the process, we used some LIF data proposed in \cite{TekaMarinovSantamaria2014}.
\begin{figure}[t]
\centering
{\includegraphics[width=0.49\textwidth]{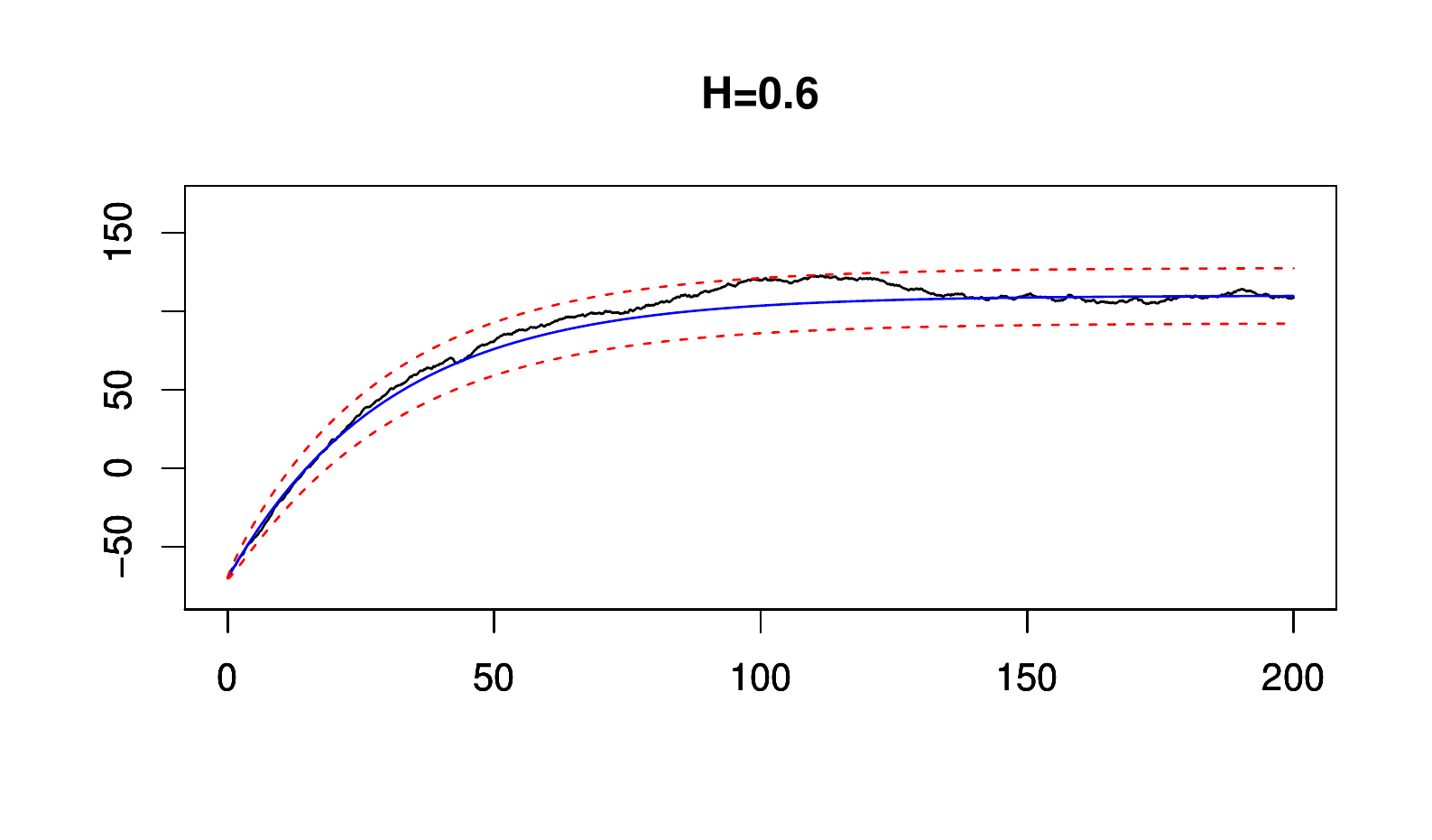}}
{\includegraphics[width=0.49\textwidth]{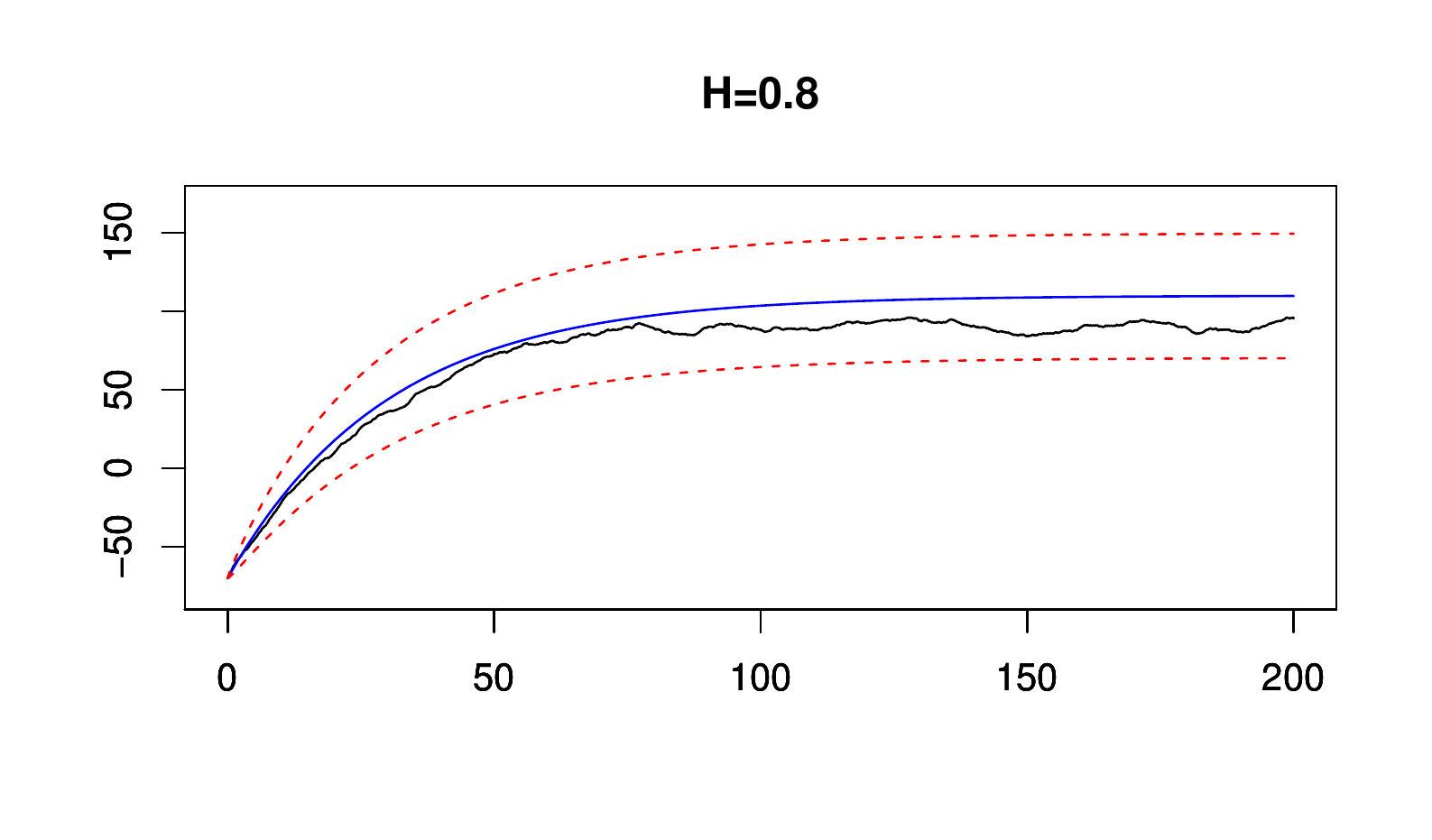}}
\caption{Sample paths of $V_t$ with constant stimulus for $\widetilde{V}=-70$, $\theta=30$,$\sigma=1$, $I_0=6 $ and different values of $H$: the thick blue line is the expectation function while the dashed red lines delimit a statistical confidence of $99\%$ for fixed time.}
\label{Simpath1}
\end{figure}
In Figure \ref{Simpath1} we have shown some simulated sample paths for the process $V_t$ with a non-random constant forcing term $ I$. The expectation function is given by equation \eqref{eq:meanconst} while the variance is given by equation \eqref{eq:detvar}. As we expected, the process follows the shape of its expectation function and stabilizes itself near the asymptotic value $\widetilde{V}+I_0\theta$. One can imagine this behavior is due to the constant stimulus received by the process.
\begin{figure}[t]
\centering
{\includegraphics[width=0.49\textwidth]{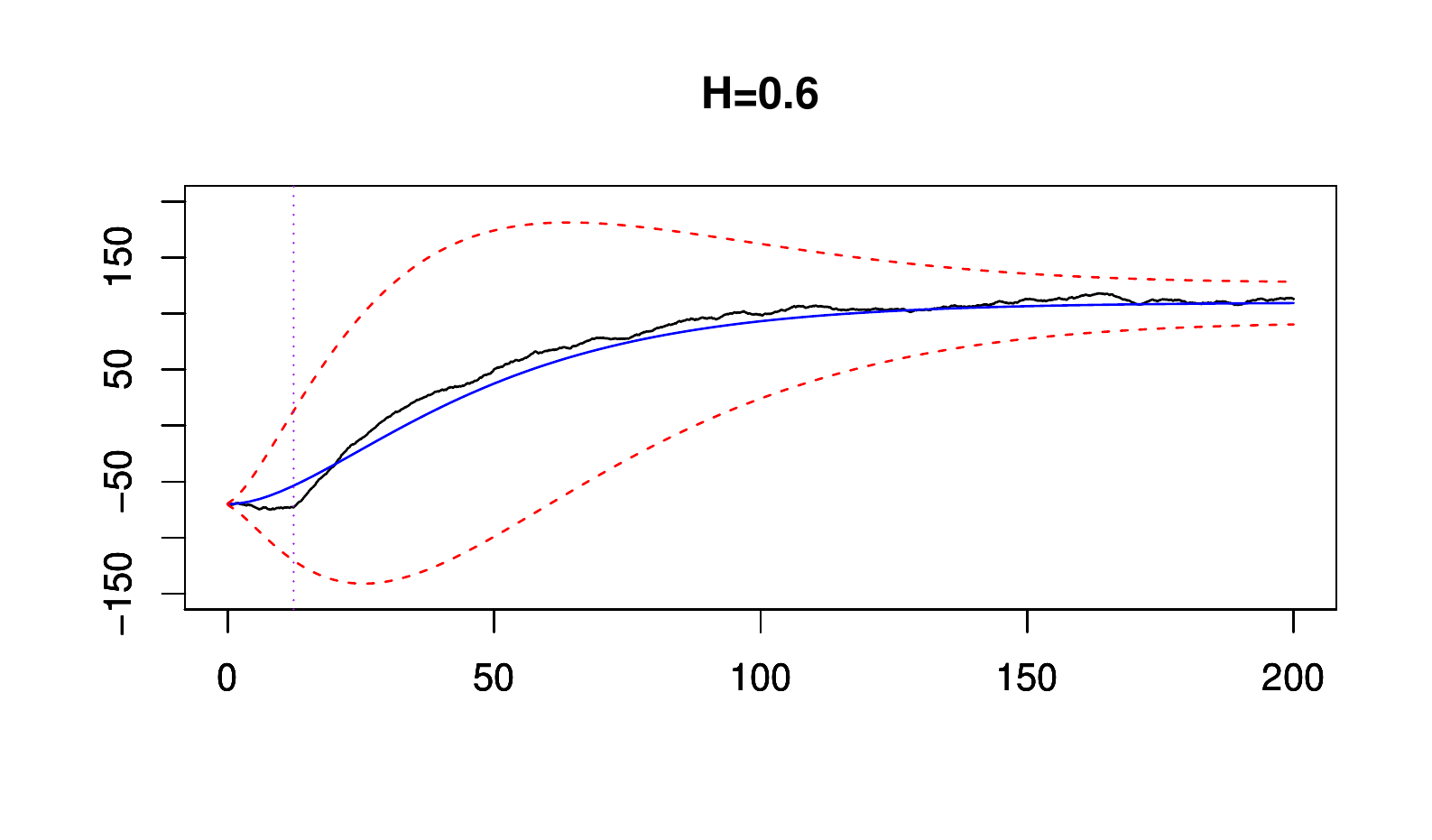}}
{\includegraphics[width=0.49\textwidth]{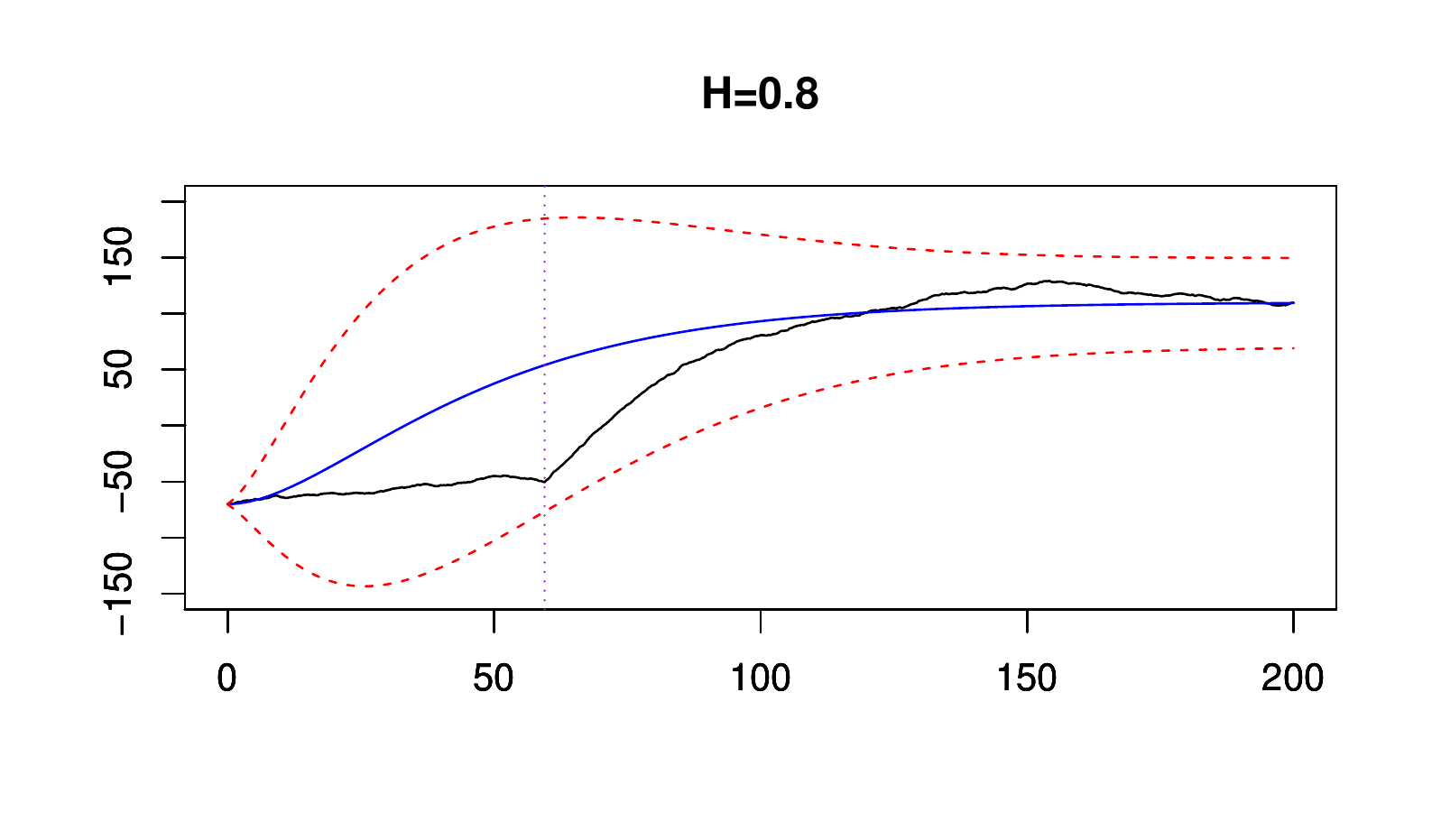}}
\caption{Sample paths of $V_t$ with stochastic stimulus with activation time $T$ whose distribution is exponential for $\widetilde{V}=-70$, $\theta=30$, $\ttau=20$,$\sigma=1$, $I_0=6$ and different values of $H$: the thick blue line is the expectation function, the dashed red lines delimit a statistical confidence of $99\%$ for fixed time and the purple line indicates the activation time of the current.}
\label{Simpath3}
\end{figure}
In Figure \ref{Simpath3} we have simulated the process $V_t$ for a forcing term $ I$ defined on a different probability space $(\Omega', \cF', \bP')$ as described in subsection $3.1$ for $n=1$ and $T$ exponential random variable on $(\Omega', \cF', \bP')$. Simulating this process, one can see how $V_t$ drastically changes its behaviour after the activation time. In particular it starts stable near $\widetilde{V}$ and then, after the activation time, it follows the new expectation function (given in equation \eqref{eq:expstoctermot}) as $t$ grows.\\

If we want to simulate the process up to a non-non-random stopping time, we cannot use Circulant Embedding method to simulate the increments of $B^H_t$, because we need to know the interval in which we are going to simulate the process. Circulant Embedding method is a fast and exact method for simulation of stationary Gaussian processes, in particular for increments of the fBm (see \cite{PerrinHarbaJennaneIribarren2002}). The construction of the circulant matrix in which we want to embed our covariance matrix depends on the knowledge of the whole covariance matrix, which means we need to know how much nodes we want to simulate.\\
To solve this problem one can use Cholesky factorization method, as explained in \cite{AsmussenGlynn2007}, to simulate $B^H_t$ or its increments. It is a slower method, but allows us to dynamically chose the stopping time. This simulation method is based on a recursive construction, so one can dynamically update the number of nodes creating a new one from the already known previous nodes. Moreover, the simulation algorithms we provided are recursive formulas, so they can be used while dynamically updating the fBm. In this way we have provided a method to simulate first passage times.\\

Future works will focus on the study of first passage times for the process $V_t$ through constant thresholds. We finally remark that the simulation tool will allow a more detailed and extensive validation analysis of the proposed model for different choices of the stochastic forcing process.


\bibliographystyle{plain}      
\bibliography{biblionew}   

\end{document}